\documentclass[12 pt]{amsart}

\usepackage{url}
\usepackage{amssymb}
\usepackage{amsmath}
\usepackage{stmaryrd}
\usepackage{siunitx}
\usepackage{commath}
\usepackage{graphicx}
\usepackage[caption=false]{subfig}
\usepackage{enumerate}
\usepackage{enumitem}
\usepackage{todonotes}
\usepackage{pgfplotstable, pgfplots}
\usepackage{cleveref}

\newtheorem{assumption}{Assumption}
\newtheorem{remark}{Remark}
\newtheorem{theorem}{Theorem}
\newtheorem{corollary}{Corollary}
\newtheorem{lemma}{Lemma}
\makeatletter
\newcommand{\tnorm}{\@ifstar\@tnorms\@tnorm}
\newcommand{\@tnorms}[1]{%
  \left|\mkern-1.5mu\left|\mkern-1.5mu\left|
   #1
  \right|\mkern-1.5mu\right|\mkern-1.5mu\right|
}
\newcommand{\@tnorm}[2][]{%
  \mathopen{#1|\mkern-1.5mu#1|\mkern-1.5mu#1|}
  #2
  \mathclose{#1|\mkern-1.5mu#1|\mkern-1.5mu#1|}
}
\makeatother
\newcommand{\jump}[1]{\llbracket #1 \rrbracket}
\newcommand{\divergence}{\operatorname{div}}

\title[Pressure-robust EDG for the Stokes problem]{A pressure-robust embedded discontinuous Galerkin method for
  the Stokes problem by reconstruction operators}

\thanks{
  PL has been funded by the Austrian Science Fund (FWF) through the research programm ``Taming complexity in partial differential systems'' (F65) - project ``Automated discretization in multiphysics'' (P10).  SR gratefully acknowledges support from the Natural Sciences and
      Engineering Research Council of Canada through the Discovery
      Grant program (RGPIN-05606-2015).}
\author[P.~L.~Lederer]{Philip L. Lederer}
\address{Institute for Analysis and Scientific Computing, TU Wien,
Wiedner Hauptstra{\ss}e 8-10, 1040 Wien, Austria}
\email{philip.lederer@tuwien.ac.at}

\author[S.~Rhebergen]{Sander Rhebergen}
\address{Department of Applied Mathematics,
    University of Waterloo, Canada}
\email{srheberg@uwaterloo.ca, http://orcid.org/0000-0001-6036-0356}

\begin{document}
\begin{abstract}
  The embedded discontinuous Galerkin (EDG) finite element method for
  the Stokes problem results in a point-wise divergence-free
  approximate velocity on cells. However, the approximate velocity is
  not $H(\divergence)$-conforming and it can be shown that this is the
  reason that the EDG method is not pressure-robust, i.e., the error
  in the velocity depends on the continuous pressure. In this paper we
  present a local reconstruction operator that maps discretely
  divergence-free test functions to exactly divergence-free test
  functions. This local reconstruction operator restores
  pressure-robustness by only changing the right hand side of the
  discretization, similar to the reconstruction operator recently
  introduced for the Taylor--Hood and mini elements by Lederer et
  al. (SIAM J. Numer. Anal., 55 (2017), pp. 1291--1314). We present an
  a priori error analysis of the discretization showing optimal
  convergence rates and pressure-robustness of the velocity
  error. These results are verified by numerical examples. The
  motivation for this research is that the resulting EDG method
  combines the versatility of discontinuous Galerkin methods with the
  computational efficiency of continuous Galerkin methods and accuracy
  of pressure-robust finite element methods.
\end{abstract}
\keywords{
  Stokes equations, embedded, discontinuous Galerkin finite element
  methods, pressure-robustness, exact divergence-free velocity
  reconstruction.}


\maketitle

\section{Introduction}
\label{sec:introduction}

Changing the body force of the continuous Stokes equations by a
gradient field changes only the pressure solution, not the velocity. A
finite element method for the Stokes equations that mimics this
property at the discrete level is called \emph{pressure-robust} and
results in an \emph{a priori} error estimate for the velocity that
does not depend on the pressure error scaled by the inverse of the
viscosity. The significance of this result is that large errors in the
pressure, as may occur for example in natural convection problems, do
not affect the velocity \cite{John:2017}.

A finite element for the Stokes equations that is both conforming and
divergence-free is pressure-robust \cite{Guzman:2014b, Guzman:2014a,
  Scott:1985}. These finite element methods, however, are generally
difficult to implement and traditional finite element methods often
relax one, or both, of these conditions. Unfortunately, the resulting
method is often not pressure-robust. It was observed in
\cite{Linke:2012, Linke:2014} that this is due to a lack of
$L^2$-orthogonality between irrotational and discretely
divergence-free vector fields. In \cite{Linke:2014}
$L^2$-orthogonality between irrotational and discretely
divergence-free vector fields is restored for the first-order
Crouzeix--Raviart element \cite{Crouzeix:1973} by replacing discretely
divergence-free vector fields by divergence-free lowest-order
Raviart--Thomas \cite{Boffi:book} velocity reconstructions wherever
$L^2$ scalar products occur in the momentum balance equations. This
simple modification (for the Stokes equations only the right-hand side
of the discretization needs to be modified locally) results in a
pressure-robust first-order Crouzeix--Raviart discretization of the
(Navier--)Stokes equations. For discretizations of the
(Navier--)Stokes equations using a `discontinuous' pressure
approximation this modification is generalized to nonconforming and
conforming mixed finite elements of arbitrary order in
\cite{Linke:2016c}. Pressure-robustness is restored for the
Taylor--Hood \cite{Hood:1974} and mini elements \cite{Arnold:1984},
which have `continuous' pressure approximations, in
\cite{Lederer:2017}.

An alternative to the above mentioned modification of traditional
finite element methods is to use $H(\divergence)$-conforming
discontinuous Galerkin (DG) methods \cite{Cockburn:2007b,
  Kanschat:2014}. $H(\divergence)$-conforming DG discretizations are
not only (automatically) pressure-robust, they are also ideally suited
for convection dominated flows due to the natural incorporation of
upwinding at element boundaries.

Unfortunately, DG methods are known to be computationally
expensive. Hybridizable discontinuous Galerkin (HDG) methods were
introduced to address this issue \cite{Cockburn:2009a}. This is
achieved by introducing new trace unknowns. The governing equations
are then posed cell-wise in terms of the approximate fields on a cell
and numerical fluxes. The numerical fluxes are defined in terms of the
traces of the approximate fields on the cell and the new trace
unknowns in such a way that the approximate fields defined on a cell
communicate only to fields that are defined on facets. This definition
of the numerical flux allows cheap elimination of all cell
degrees-of-freedom (DOFs), significantly reducing the number of
globally coupled DOFs. Recent years has seen the development of many
$H(\divergence)$-conforming HDG methods \cite{Fu:2019,
  Lehrenfeld:2016, Rhebergen:2018a} which, like the
$H(\divergence)$-conforming DG methods, are (automatically)
pressure-robust. To reduce the number of coupled DOFs even further,
$H(\divergence)$-conformity was introduced only in a relaxed manner in
\cite{Lederer:2017, Lederer:2019}. Finally, we want to mention the work
in \cite{Lederer:2019c, Lederer:2019b, lederer2019mass} where the
authors derived a mixed method with $H(\divergence)$-conforming
velocities.

The $H(\divergence)$-conforming HDG method introduced in
\cite{Rhebergen:2018a} introduces \emph{discontinuous} trace velocity
and trace pressure approximations. An alternative is to use
\emph{continuous} trace velocity and \emph{discontinuous} trace
pressure approximations. This results in the recently introduced
embedded-hybridized discontinuous Galerkin (EDG-HDG) method
\cite{Rhebergen:2020}. The $H(\divergence)$-conforming EDG-HDG method
has even less globally coupled DOFs than an HDG method due to the use
of a continuous trace velocity approximation. It is possible to lower
the number of globally coupled DOFs even further by using both
continuous trace velocity and trace pressure approximations. The
resulting method is known as an embedded discontinuous Galerkin (EDG)
method and was introduced for the Navier--Stokes equations in
\cite{labeur:2007, Labeur:2012}. It was demonstrated in
\cite{Rhebergen:2020}, using the preconditioner of
\cite{Rhebergen:2018b}, that CPU time and iteration count to
convergence is significantly reduced using continuous trace
approximations compared to using discontinuous trace approximations
(HDG method). Unfortunately, the EDG method is not pressure-robust.

In this paper we restore pressure-robustness of the EDG discretization
of the Stokes equations \cite{Labeur:2012, Rhebergen:2017}. As with
the HDG method, cell DOFs can be eliminated cheaply resulting in a
global system only for the velocity and pressure trace
approximations. Due to the continuity of the pressure trace
approximation we follow a similar approach as presented in
\cite{Lederer:2017} to restore pressure-robustness. Therein a
reconstruction operator for weakly divergence-free velocities was
defined which was based on solving local problems on vertex
patches. This local approach was motivated by the lifting techniques
introduced for equilibrated error estimators \cite{Braess:2008}. In
contrast to \cite{Lederer:2017}, where the reconstruction operator was
used to eliminate the local divergence on each cell, the
reconstruction operator in this work only enforces exact normal
continuity since the EDG solution is already exactly divergence-free
on a cell.

This paper is organized as follows. We present the EDG method for the
Stokes problem in \cref{sec:edgstokes}. \Cref{sec:pressurerobust}
presents the main result of this paper; a reconstruction operator to
restore pressure-robustness of the EDG method and an \emph{a priori}
error analysis. Numerical examples to support our theory are presented
in \cref{sec:numerical} and conclusions are drawn in
\cref{sec:conclusions}.

\section{The embedded discontinuous Galerkin method}
\label{sec:edgstokes}

In this section we present the embedded discontinuous Galerkin (EDG)
method for the Stokes problem. We introduce the approximation spaces,
the discrete problem, and discuss some properties of the discrete
Stokes problem.

\subsection{The Stokes problem}

Let $\Omega \subset \mathbb{R}^d$ be a polygonal ($d=2$) or polyhedral
($d=3$) domain and let $\partial\Omega$ denote its boundary. The
Stokes problem is given by: find the velocity $u:\Omega\to\mathbb{R}^d$
and (kinematic) pressure $p:\Omega \to \mathbb{R}$ such that
\begin{subequations}
  \begin{align}
    -\nu \nabla^2u + \nabla p &= f && \text{in } \Omega,
    \\
    \nabla \cdot u &= 0 && \text{in } \Omega,
    \\
    u &= 0 && \text{on } \partial\Omega,
    \\
    \int_{\Omega} p \dif x &= 0,
  \end{align}
  \label{eq:stokes}
\end{subequations}
where $f:\Omega \to \mathbb{R}^d$ is a given a body force and
$\nu \in \mathbb{R}^+$ is the kinematic viscosity.

\subsection{The discrete Stokes problem}
\label{ss:discreteStokesProblem}

To discretize the Stokes problem \cref{eq:stokes} by the EDG method,
we first introduce a triangulation $\mathcal{T} := \cbr{K}$ of
$\Omega$ consisting of non-overlapping cells $K$. We denote the
boundary of a cell $K$ by $\partial K$ and the outward unit normal
vector on $\partial K$ by $n$. The diameter of a cell $K$ is denoted
by $h_K$ and we define $h := \max_{K\in\mathcal{T}}h_K$. Two adjacent
cells $K^+$ and $K^-$ share an interior facet $F$, while a boundary
facet is part of $\partial K$ that lies on the domain boundary
$\partial \Omega$. The set of all facets is denoted by
$\mathcal{F} := \cbr{F}$. We denote the union of all facets by
$\Gamma^0$. 

We require the following approximation spaces:
\begin{subequations}
  \begin{align}
    V_h &:= \cbr[1]{v_h \in \sbr[0]{L^2(\Omega)}^d \ : \ v_h \in \sbr{P_k(K)}^d,\ \forall K\in\mathcal{T}},
    \\
    Q_h &:= \cbr[1]{q_h \in L^2(\Omega) \ : \ q_h \in P_{k-1}(K),\ \forall K\in\mathcal{T}},
    \\
    \bar{V}_h^d &:= \cbr[1]{\bar{v}_h \in \sbr[0]{L^2(\Gamma^0)}^d \ : \ \bar{v}_h \in \sbr{P_k(F)}^d,\ \forall F\in\mathcal{F},\ \bar{v}_h=0 \text{ on } \partial\Omega},
    \\
    \label{eq:approx_spaces_d}
    \bar{Q}_h^{m,d} &:= \cbr[1]{\bar{q}_h \in L^2(\Gamma^0) \ : \ q_h \in P_{m}(F),\ \forall F\in\mathcal{F}},
  \end{align}
  \label{eq:approx_spaces}
\end{subequations}
where $P_k(K)$ and $P_k(F)$ denote the set of polynomials of degree
$k$ on a cell $K$ and on a facet $F$, respectively. In this manuscript
we consider both the case where $m=k-1$ (for $k\ge 2$) and $m=k$ (for
$k\ge 1$). 

We next impose continuity of the `facet' spaces:
\begin{equation}
  \label{eq:approx_spaces_edg}
  \bar{V}_h := \bar{V}_h^d \cap \sbr[1]{C^0(\Gamma^0)}^d
  \quad \text{and} \quad
  \bar{Q}_h^m := \bar{Q}_h^{m,d} \cap C^0(\Gamma^0),
\end{equation}
and define $\boldsymbol{V}_h := V_h \times \bar{V}_h$ and
$\boldsymbol{Q}_h^m := Q_h \times \bar{Q}_h^m$. Function pairs in
$\boldsymbol{V}_h$ and $\boldsymbol{Q}_h^m$ are denoted by
$\boldsymbol{v}_h := (v_h, \bar{v}_h) \in \boldsymbol{V}_h$ and
$\boldsymbol{q}_h := (q_h, \bar{q}_h) \in \boldsymbol{Q}_h^m$. For
notational purposes we will drop the superscript $m$ in the definition
of the pressure space for the remainder of this paper if a result
holds for both $m=k-1$ and $m=k$.

The EDG method for the Stokes problem \cref{eq:stokes} is given by
\cite{Rhebergen:2020}: find
$(\boldsymbol{u}_h,\boldsymbol{p}_h) \in \boldsymbol{V}_h \times
\boldsymbol{Q}_h$, such that
\begin{subequations}
  \label{eq:discrete_stokes}
  \begin{align}
    \label{eq:discrete_stokes_mom}
    a_h(\boldsymbol{u}_h, \boldsymbol{v}_h) + b_h(\boldsymbol{p}_h, v_h) &= \int_{\Omega} f \cdot v_h \dif x
    && \forall \boldsymbol{v}_h \in \boldsymbol{V}_h,
    \\
    \label{eq:discrete_stokes_mass}
    b_h(\boldsymbol{q}_h, u_h) &= 0
    && \forall \boldsymbol{q}_h \in \boldsymbol{Q}_h,
  \end{align}
\end{subequations}
where
\begin{subequations}
  \begin{align}
    \label{eq:formA}
    a_h(\boldsymbol{u}, \boldsymbol{v}) :=&
    \sum_{K\in\mathcal{T}}\int_{K}\nu \nabla u : \nabla v \dif x
    + \sum_{K\in\mathcal{T}}\int_{\partial K}\frac{\nu \alpha}{h_K}(u - \bar{u})\cdot(v - \bar{v}) \dif s
    \\
    \nonumber
    &- \sum_{K\in\mathcal{T}}\int_{\partial K}\nu \sbr[1]{(u-\bar{u})\cdot \pd{v}{n}
    + \pd{u}{n}\cdot(v-\bar{v})} \dif s,
    \\
    \label{eq:formB}
    b_h(\boldsymbol{p}, v) :=&
    - \sum_{K\in\mathcal{T}}\int_{K}p \nabla \cdot v \dif x
    + \sum_{K\in\mathcal{T}}\int_{\partial K}v\cdot n\bar{p} \dif s,
  \end{align}
  \label{eq:bilin_forms}
\end{subequations}
with $\alpha > 0$ a penalty parameter that needs to be chosen
sufficiently large to ensure stability \cite{Rhebergen:2017,
  Wells:2011}.

\subsection{Properties of the discrete Stokes problem}
\label{prop::discstokes}

To discuss properties of the discrete Stokes problem we require the
following extended function spaces:
\begin{align*}
  V(h) &:= V_h + \sbr[1]{H_0^1(\Omega)}^d \cap \sbr[1]{H^2(\Omega)}^d,
  &
    Q(h) &:= Q_h + L_0^2(\Omega) \cap H^1(\Omega),
  \\
  \bar{V}(h) &:= \bar{V}_h + [H_0^{3/2}(\Gamma_0)]^d,
  &
  \bar{Q}(h) &:= \bar{Q}_h + H^{1/2}_0(\Gamma_0).  
\end{align*}
We define on $V(h) \times \bar{V}(h)$ the mesh-dependent norms
\begin{equation}
  \tnorm{ \boldsymbol{v} }_{1}^2 := \sum_{K\in\mathcal{T}}\norm{\nabla v }^2_{K}
  + \sum_{K\in\mathcal{T}} \frac{\alpha_{v}}{h_K}\norm{\bar{v} - v}^2_{\partial K},
  \qquad
  \tnorm{ \boldsymbol{v} }_{1,*}^2 := \tnorm{ \boldsymbol{v} }_*^2
  + \sum_{K\in\mathcal{T}}\frac{h_K}{\alpha}\norm{\frac{\partial v}{\partial n}}^2_{\partial K},
\end{equation}
and remark that the norms $\tnorm{\cdot}_1$ and $\tnorm{\cdot}_{1,*}$
are equivalent on the finite element space $\boldsymbol{V}_h$. On
$Q(h) \times \bar{Q}(h)$ we introduce the norm
\begin{equation}
  \label{eq:def_normbarq}
    \tnorm{\boldsymbol{q}}^2_p := \norm{q}^2_{\Omega} + \norm{\bar{q}}_{p}^2,
\end{equation}
where $\norm{\bar{q}}_{p}^2 := \sum_{K\in\mathcal{T}}h_K\norm{\bar{q}}^2_{\partial K}$.

Boundedness and stability of $a_h$ was proven in
\cite{Rhebergen:2017}. In particular, it was shown that there exists a
constant $\alpha_0>0$ such that for $\alpha > \alpha_0$
\begin{equation}
  \label{eq:stab_ah}
  a_h(\boldsymbol{v}_h, \boldsymbol{v}_h) \gtrsim \nu \tnorm{\boldsymbol{v}_h}^2_1 \qquad \forall \boldsymbol{v}_h \in \boldsymbol{V}_h,
\end{equation}
and that
\begin{equation}
  \label{eq:bound_ah}
  \envert{a_h(\boldsymbol{u}, \boldsymbol{v}_h)}
  \lesssim \nu \tnorm{\boldsymbol{u}}_{1,*} \tnorm{\boldsymbol{v}_h}_1 \qquad
  \forall \boldsymbol{u}\in V(h) \times \bar{V}(h) \textrm{ and }
  \forall \boldsymbol{v}_h \in \boldsymbol{V}_h.
\end{equation}
Furthermore, stability of $b_h$ was shown in \cite{Rhebergen:2020}:
\begin{equation}
  \label{eq:stab_bh}
  \tnorm{\boldsymbol{q}_h}_{p} \lesssim \sup_{\boldsymbol{v}_h \in \boldsymbol{V}_{h}}
  \frac{ b_h(\boldsymbol{q}_h, v_h) }{\tnorm{ \boldsymbol{v}_h }_{1}} \qquad
  \forall \boldsymbol{q}_h \in \boldsymbol{Q}_h.
\end{equation}
The well-posedness of the discrete Stokes problem
\cref{eq:discrete_stokes} follows directly from the above stability
results (see for example \cite{Boffi:book}).

Setting $\boldsymbol{q}_h = (\nabla\cdot u_h, 0) \in \boldsymbol{Q}_h$
in \cref{eq:discrete_stokes_mass} it is immediately clear that
$\nabla \cdot u_h = 0$ point-wise on a cell. The EDG method, however,
is not $H(\divergence)$-conforming on $\Omega$. This is because
\cref{eq:discrete_stokes_mass} imposes only weak continuity of the
normal component of the velocity across cell facets. The lack of
$H(\divergence)$-conformity was shown in \cite[Remark
1]{Rhebergen:2020} to be the reason that the EDG method is not
pressure-robust; the velocity error has a dependence on $1/\nu$ times
the pressure error,
\begin{equation}
  \label{eq:notpressurerob}
  \tnorm{\boldsymbol{u} - \boldsymbol{u}_h}_1
  \lesssim \inf_{\substack{\boldsymbol{v}_h \in \boldsymbol{V}_h \\b_{h}(\boldsymbol{q}_{h}, v_{h}) = 0
      \, \forall \boldsymbol{q}_{h} \in \boldsymbol{Q}_h }} \tnorm{\boldsymbol{u} - \boldsymbol{v}_h}_{1,*}
  + \frac{1}{\nu} \inf_{\boldsymbol{q}_h \in \boldsymbol{Q}_h}\tnorm{\boldsymbol{p}
    - \boldsymbol{q}_h}_p.
\end{equation}
In the remainder of this paper we modify the EDG method
\cref{eq:discrete_stokes} to restore pressure-robustness.

\section{Pressure-robustness}
\label{sec:pressurerobust}

In the continuous Stokes problem the velocity field is not affected by
adding a gradient field to the body force. To see this, consider the
continuous Stokes problem: find $u \in H^1_0(\Omega)$ and
$p \in L^2_0(\Omega)$ such that
\begin{subequations}
  \label{eq:continuousStokes}
  \begin{align}
    \label{eq:continuousStokes_mom}
    \int_{\Omega} \nu \nabla u : \nabla v \dif x - \int_{\Omega} p \nabla \cdot v \dif x
    &= \int_{\Omega} f \cdot v \dif x && \forall v \in H^1_0(\Omega),
    \\
    \label{eq:continuousStokes_mass}
    -\int_{\Omega}q \nabla \cdot u \dif x &=0 && \forall q \in L^2_0(\Omega).
  \end{align}  
\end{subequations}
By defining
$V_0 := \cbr[0]{v \in H^1_0(\Omega)\ :\ \int_{\Omega}q\nabla\cdot v =
  0 \ \forall q\in L^2_0(\Omega)}$ we can formulate the following
equivalent problem to \cref{eq:continuousStokes}: find $u \in V_0$
such that
\begin{equation}
  \label{eq:continuousStokes_equiv}
  \int_{\Omega} \nu \nabla u : \nabla v \dif x = \int_{\Omega} f \cdot v \dif x
  \quad
  \forall v \in V_0.
\end{equation}
Changing the body force by a gradient field, i.e., changing $f$ to
$f + \nabla\psi$ with $\psi \in H^1(\Omega)\cap L^2_0(\Omega)$, we
observe that
\begin{equation}
  \label{eq:continuousStokes_equiv_grad}
  \int_{\Omega} \nu \nabla u : \nabla v \dif x
  = \int_{\Omega} (f + \nabla\psi) \cdot v \dif x
  = \int_{\Omega} f \cdot v \dif x
  \quad
  \forall v \in V_0.
\end{equation}
In other words, the irrotational part of the body force does not
affect the velocity solution.

Many traditional finite element methods are not pressure-robust
because the irrotational part of the body force is not
$L^2$-orthogonal with discretely divergence-free vector fields
\cite{Linke:2014}. This is true also for the EDG method
\cref{eq:discrete_stokes}. Indeed, define
$V_{h,0} := \cbr[0]{v_h \in V_h\ :\ b_h(\boldsymbol{q}_h, v_h) = 0\
  \forall \boldsymbol{q}_h \in \boldsymbol{Q}_h}$. Then an equivalent
problem to \cref{eq:discrete_stokes} is given by: find
$\boldsymbol{u}_h \in V_{h,0} \times \bar{V}_h$ such that
\begin{equation}
  \label{eq:discreteStokes_equiv}
  a_h(\boldsymbol{u}_h, \boldsymbol{v}_h) = \int_{\Omega} f \cdot v_h \dif x
  \quad
  \forall \boldsymbol{v}_h \in V_{h,0} \times \bar{V}_h.
\end{equation}
Adding a gradient field $\nabla\psi$, with
$\psi \in H^1(\Omega)\cap L^2_0(\Omega)$, to the body force, we
observe that for all $\boldsymbol{v}_h \in V_{h,0} \times \bar{V}_h$
it holds that
\begin{equation}
  \label{eq:discreteStokes_equiv_grad_1}
  \begin{split}
    a_h(\boldsymbol{u}_h, \boldsymbol{v}_h)
    &= \int_{\Omega} (f + \nabla\psi) \cdot v_h \dif x
    \\
    &= \int_{\Omega} f \cdot v_h \dif x
    - \sum_{K\in\mathcal{T}}\int_{K} \psi \nabla \cdot v_h \dif x
    + \sum_{K\in\mathcal{T}}\int_{\partial K} \psi v_h \cdot n \dif s.
  \end{split}
\end{equation}
As we saw in \cref{prop::discstokes}, if $v_h \in V_{h,0}$ then $v_h$
is exactly divergence-free on a cell $K$, but $v_h$ is not
$H(\divergence)$-conforming on $\Omega$. We obtain
\begin{equation}
  \label{eq:discreteStokes_equiv_grad_2}
  a_h(\boldsymbol{u}_h, \boldsymbol{v}_h)
  = \int_{\Omega} f \cdot v_h \dif x
  + \sum_{K\in\mathcal{T}}\int_{\partial K} \psi v_h \cdot n \dif s,
\end{equation}
showing that the irrotational part of the body force now changes also
the discrete velocity $\boldsymbol{u}_h$. If $v_h$ were to have been
$H(\divergence)$-conforming on $\Omega$, such as the HDG and EDG-HDG
variants of \cref{eq:discrete_stokes} \cite{Rhebergen:2020}, then the
last term in \cref{eq:discreteStokes_equiv_grad_2} would have vanished
and the discretization would be pressure-robust.

In the following sections we modify the EDG method to restore
pressure-robustness. For this we require the following notation. The
set of vertices is denoted by $\mathcal{V}$. For each vertex
$V \in \mathcal{V}$ we define the vertex patch
$\omega_V := \cup_{K:V\in K} K \subset \Omega$ and the triangulation
on the vertex patch
$\mathcal{T}_V := \cbr{K\in \mathcal{T}\ :\ K\cap \omega_V \ne
  \emptyset}$. We denote by $\mathcal{F}_V$ the set of all interior
facets in $\omega_V$ with respect to the triangulation
$\mathcal{T}_V$. The union of all facets in $\mathcal{F}_V$ is denoted
by $\Gamma^0_V$.

\subsection{A pressure-robust EDG method}
\label{ss:pressurerobustEDG}

To restore pressure-robustness of the EDG method
\cref{eq:discrete_stokes}, we follow an idea first introduced in
\cite{Linke:2014} for the Crouzeix--Raviart finite element. Therein a
reconstruction operator $\mathcal{R}$ is introduced that maps weakly
divergence-free velocities onto exactly divergence-free velocities,
i.e., velocities that are exactly divergence-free on a cell and
$H(\divergence)$-conforming. It is then proposed to replace the test
function on the right hand side of \cref{eq:discrete_stokes} by the
reconstruction operator applied to the test function.

The reconstruction operator developed for the Crouzeix--Raviart finite
element, however, cannot directly be applied to the EDG
discretization. Depending on the continuity properties of the pressure
approximation, several reconstruction operators have been
introduced. For discontinuous pressure approximations, the
reconstruction operator can be defined locally on cells (see
\cite{Linke:2014, Linke:2016c}). For continuous pressure
approximations, Lederer et al. \cite{Lederer:2017b} defined a
reconstruction operator on vertex patches based on the ideas of the
equilibrated a posteriori error estimator \cite{Braess:2008}. This
reconstruction operator was successfully applied to the Taylor--Hood
and mini elements.

To construct a reconstruction operator $\mathcal{R}:V_h \to V_h$ for
the EDG method \cref{eq:discrete_stokes}, in which the approximate
trace pressure is continuous, we follow a similar approach as
\cite{Lederer:2017b}. However, where in \cite{Lederer:2017b} a weakly
divergence-free velocity is $H(\divergence)$-conforming, but not
exactly divergence-free on a cell, the opposite is true for the EDG
method. As such, the reconstruction operator for the EDG method must
be such that it compensates the normal jumps of a discrete velocity on
cell-boundaries without changing the divergence of a discrete velocity
on a cell. The definition and analysis of such an operator is
postponed to \cref{ss:reconstruction}. To define a pressure-robust EDG
method, it is sufficient for now to assume the following:

\begin{assumption}[Properties of an EDG reconstruction operator]
  \label{as:R_assumption}
  There exists a reconstruction operator
  $\mathcal{R}: V_h \to V_h$ such that
  \begin{enumerate}[label=(\roman*.)]
  \item If $u_h \in V_h$ such that $b_h(\boldsymbol{q}_h, u_h) = 0$
    for all $\boldsymbol{q}_h \in \boldsymbol{Q}_h$, then
    $\nabla \cdot \mathcal{R}(u_h) = 0$ point-wise on cells and
    $\mathcal{R}(u_h) \in H(\divergence,
    \Omega)$. \label{as:R_assumption_one}
  \item For all $u_h \in V_h$ we have
    \begin{equation*}
      \norm{ \mathcal{R}(u_h) - u_h}_{L^2(\Omega)} \lesssim h \tnorm {u_h}_{1}. 
    \end{equation*}
    \label{as:R_assumption_two}
  \item For $g \in [L^2(\Omega)]^d$ we have $(g, \mathcal{R}(u_h) - u_h)_{L^2(\Omega)}
      \lesssim \tnorm{g}_{\textrm{con},k}
       \tnorm{ u_h}_1,$
       with
       \begin{align*}
         \tnorm{g}_{\textrm{con},k} := \del[2]{\sum_{V \in \mathcal{V}} h^2 \int_{\omega_V} (g - \Pi^{k-2}_{\omega_V}g)^2 \dif x  }^{1/2},
       \end{align*}
       where $\Pi_{\omega_V}^{k-2}$ is the $L^2$-projection operator
       on polynomials of order $k-2$ on the vertex patch
       $\omega_V$. \label{as:R_assumption_three}
  \end{enumerate}  
\end{assumption}

\begin{remark}
  \label{patch_iterpolation}
  If $g \in H^{s}(\Omega)$, with $s \ge k-1$, then a standard scaling
  argument shows that
  \begin{equation*}
    \tnorm{g}_{\textrm{con},k} \lesssim h^{k} \| g \|_{H^{k-1}(\Omega)}.
  \end{equation*}
\end{remark}

Let $\mathcal{R}:V_h \to V_h$ be a reconstruction operator that
satisfies \cref{as:R_assumption}. We define the pressure-robust EDG
method as: find
$(\boldsymbol{u}_h,\boldsymbol{p}_h) \in \boldsymbol{V}_h \times
\boldsymbol{Q}_h$ such that
\begin{subequations}
  \label{eq:discrete_stokes_pr}
  \begin{align}
    \label{eq:discrete_stokes_mom_pr}
    a_h(\boldsymbol{u}_h, \boldsymbol{v}_h) + b_h(\boldsymbol{p}_h, v_h) &= \int_{\Omega} f \cdot \mathcal{R}(v_h) \dif x
    && \forall \boldsymbol{v}_h \in \boldsymbol{V}_h,
    \\
    \label{eq:discrete_stokes_mass_pr}
    b_h(\boldsymbol{q}_h, u_h) &= 0
    && \forall \boldsymbol{q}_h \in \boldsymbol{Q}_h.
  \end{align}
\end{subequations}
Since by \cref{as:R_assumption_one} of \cref{as:R_assumption}
$\mathcal{R}(v_h)$ is exactly divergence-free, it is clear that the
irrotational part of the body force will not change the discrete
velocity $\boldsymbol{u}_h$. 

\subsection{A priori error analysis}
\label{ss:apriori}

In this section we show optimal a priori velocity error estimates for
the EDG method \cref{eq:discrete_stokes_pr} that are independent of
the pressure.

\begin{theorem}[Pressure-robust a priori velocity error estimate]
  \label{th:bestapprox}
  Let
  $(\boldsymbol{u}_h, \boldsymbol{p}_h) \in \boldsymbol{V}_h \times
  \boldsymbol{Q}_h$ be the solution to \cref{eq:discrete_stokes_pr},
  let
  $(u,p) \in [H^{2}(\mathcal{T}) \cap H^1_0(\Omega)]^d \times
  L^2_0(\Omega)$ be the exact solution to \cref{eq:stokes}, and set
  $\boldsymbol{u}=(u,u)$. The following a priori error estimate holds:
  \begin{align}
    \label{eq:errorestimate_pr}
    \tnorm{\boldsymbol{u} - \boldsymbol{u}_h}_1 \lesssim
    \inf\limits_{\substack{\boldsymbol{v}_h \in \boldsymbol{V}_h \cap H(\divergence) \\\divergence(v_h) = 0 } }
    \tnorm{\boldsymbol{u} - \boldsymbol{v}_h}_{1,*}
    +\tnorm{u}_{\emph{\textrm{con}},k}.
  \end{align}
\end{theorem}

\begin{proof}
  Let $\boldsymbol{v}_h = (v_h, \bar{v}_h) \in \boldsymbol{V}_h$ be an
  arbitrary test function such that
  $v_h \in V_h \cap H(\divergence, \Omega)$ and $\nabla \cdot v_h = 0$
  and let $\boldsymbol{w}_h := \boldsymbol{v}_h -
  \boldsymbol{u}_h$. By stability \cref{eq:stab_ah} and boundedness
  \cref{eq:bound_ah} of $a_h$ we find
  \begin{equation}
    \label{eq:pressrobboundingstab}
    \begin{split}
    \nu \tnorm{\boldsymbol{w}_h}_1
    &\lesssim a_h(\boldsymbol{w}_h, \boldsymbol{w}_h)
    \\ 
    &= a_h(\boldsymbol{v}_h - \boldsymbol{u}, \boldsymbol{w}_h) + a_h(\boldsymbol{u} - \boldsymbol{u}_h, \boldsymbol{w}_h)
    \\ 
    &\lesssim \nu \tnorm{\boldsymbol{v}_h - \boldsymbol{u}}_{1,*} \tnorm{ \boldsymbol{w}_h }_1
      + a_h(\boldsymbol{u}, \boldsymbol{w}_h)
      - a_h(\boldsymbol{u}_h, \boldsymbol{w}_h).      
    \end{split}
  \end{equation}
  Consider the second term on the right hand side of
  \cref{eq:pressrobboundingstab}. Continuity of $u$ and $\nabla u$
  across element interfaces and integration by parts gives
  \begin{equation}
    \label{eq:secondtermrhs}
    \begin{split}
      a_h(\boldsymbol{u} ,\boldsymbol{w}_h)
      &= \sum_{K \in \mathcal{T}} \int_K \nu \nabla u : \nabla w_h \dif x
      + \sum_{K \in \mathcal{T}} \int_{\partial K} \nu \pd{u}{n} \cdot (\bar{w}_h - w_h)\dif s \\
      &= \sum_{K \in \mathcal{T}} \int_K -\nu \Delta u \cdot w_h \dif x
      + \sum_{K \in \mathcal{T}} \int_{\partial K} \nu \pd{u}{n} \cdot \bar{w}_h \dif s
      = (-\nu \Delta u ,w_h)_{L^2(\Omega)},      
    \end{split}
  \end{equation}
  where we used that $\bar{w}_h$ is single valued on element
  interfaces. Consider now the third term on the right hand side of
  \cref{eq:pressrobboundingstab}. Since $\boldsymbol{u}_h$ satisfies
  \cref{eq:discrete_stokes_pr},
  \begin{equation}
    \label{eq:thirdtermrhs}
    \begin{split}
      a_h(\boldsymbol{u}_h,\boldsymbol{w}_h)
      &= \int_\Omega f \cdot \mathcal{R}(w_h) \dif x - b_h(\boldsymbol{p}_h, w_h)
      = \int_\Omega f \cdot \mathcal{R}(w_h) \dif x
      \\
      &= \int_\Omega (-\nu\Delta u + \nabla p) \cdot \mathcal{R}(w_h) \dif x
      = (-\nu\Delta u, \mathcal{R}(w_h))_{L^2(\Omega)},      
    \end{split}
  \end{equation}
  where we used integration by parts and \cref{as:R_assumption_one} of
  \cref{as:R_assumption} for the last equality. Combining
  \cref{eq:pressrobboundingstab}--\cref{eq:thirdtermrhs}, using
  \cref{as:R_assumption_two} and \cref{as:R_assumption_three} of
  \cref{as:R_assumption} we find: 
  \begin{align*}
    \nu\tnorm{\boldsymbol{w}_h}^2_1
    &\lesssim \nu\tnorm{\boldsymbol{v}_h - \boldsymbol{u}}_{1,*} \tnorm{\boldsymbol{w}_h}_1
      +  (-\nu \Delta u, w_h - \mathcal{R}(w_h))_{L^2(\Omega)}
    \\
    &\lesssim \nu \tnorm{\boldsymbol{v}_h - \boldsymbol{u}}_{1,*} \tnorm{\boldsymbol{w}_h}_1
      +  \nu \del[2]{\sum_{V \in \mathcal{V}} h^2 \int_{\omega_V} (\Delta u - \Pi^{k-2}_{\omega_V}\Delta u)^2 \dif x  }^{1/2}
      \tnorm{ \boldsymbol{w}_h}_1
    \\
    &\lesssim \nu \tnorm{\boldsymbol{v}_h - \boldsymbol{u}}_{1,*} \tnorm{\boldsymbol{w}_h}_1     
            + \nu \tnorm{u}_{{\textrm{con}},k}\tnorm{\boldsymbol{w}_h}_1.
  \end{align*}
  The result follows after dividing by
  $\nu\tnorm{\boldsymbol{w}_h}_1$, noting that $\boldsymbol{v}_h$ is
  arbitrary, and a triangle inequality.
\end{proof}

Observe that unlike the velocity error estimate
\cref{eq:notpressurerob} for the discrete Stokes problem
\cref{eq:discrete_stokes}, the velocity error estimate
\cref{eq:errorestimate_pr} of the pressure-robust EDG method
\cref{eq:discrete_stokes_pr} does not depend on the pressure error. A
consequence of \cref{th:bestapprox} is the following result.

\begin{corollary}
  \label{cor:ratesofconv}
  Let
  $(u,p) \in H^{2}(\mathcal{T}) \cap H^1_0(\Omega) \cap
  H^{k+1}(\Omega) \times H^k(\Omega)\cap L^2_0(\Omega)$ be the exact
  solution to \cref{eq:stokes}, and set $\boldsymbol{u}=(u,u)$.  Then
  the solution
  $(\boldsymbol{u}_h, \boldsymbol{p}_h) \in \boldsymbol{V}_h \times
  \boldsymbol{Q}_h$ to \cref{eq:discrete_stokes_pr} satisfies
  \begin{equation*}
    \tnorm{\boldsymbol{u} - \boldsymbol{u}_h}_1 \lesssim  h^{k} \norm{ u }_{H^{k+1}(\mathcal{T})}.
  \end{equation*}
\end{corollary}
\begin{proof}
  Let
  $\boldsymbol{v}_h = (I_{B\!D\!M} u, \Pi_{\mathcal{F}}u) \in
  \boldsymbol{V}_h$, where $I_{B\!D\!M}:\sbr{H^1(\Omega)}^d \to V_h$
  is the usual BDM interpolation operator (see for example \cite[Lemma
  7]{Hansbo:2002}) and $\Pi_{\mathcal{F}}u$ is the $L^2$-projection
  into the facet velocity space. Since
  $\nabla \cdot I_{B\!D\!M} u = 0$ the result follows by the
  approximation properties of the BDM interpolation operator and the
  $L^2$-projection (see \cite{Boffi:book}), \cref{th:bestapprox} and
  \cref{patch_iterpolation}.
\end{proof}

\subsection{An EDG-reconstruction operator}
\label{ss:reconstruction}

In this section we present an EDG reconstruction operator that
satisfies the properties of \cref{as:R_assumption}. Its construction
is based on the ideas of the equilibrated error estimator (see
\cite{Braess:2008}).

To construct the EDG reconstruction operator we require the following
spaces:
\begin{align*}
  \Sigma_h^V &:= \cbr{ \tau_h \in [L^2(\omega_V)]^d: \tau_h \in [P_k(K)]^d, \forall K \in \mathcal{T}_V
               \textrm{ and } \tau_h \cdot n = 0 \textrm{ on } \partial \omega_V },
  \\
  Q^V_h &:= \cbr[1]{q_h \in L^2(\omega_V) \ : \ q_h \in P_{k-1}(K),\ \forall K\in\mathcal{T}_V},
  \\
  \bar{Q}^V_h &:= \cbr[1]{\bar{q}_h \in L^2(\Gamma^0_V) \ : \ q_h \in P_{k}(F),\ \forall F\in\mathcal{F}_V},
  \\
  \Lambda_h^V &:= \kappa_{x - V}([P_{k-3}(\omega_V)]^{d(d-1)/2}),
\end{align*}
where we note that $\Sigma_h^V$ is the discontinuous BDM space on
$\mathcal{T}_V$ of degree $k$ with zero normal component on the
boundary of the vertex patch $\omega_V$. Furthermore, $\kappa$ is the
Koszul operator. For $d=2$ with $x=(x_1,x_2)$ and for $d = 3$ with
$x=(x_1,x_2,x_3)$ it is given by
\begin{align*}
  \begin{array}{lcccl}
    \kappa_{x} : L^2(\Omega) \rightarrow  [L^2(\Omega)]^2,
    &\quad &\quad&\quad & \kappa_{x} : [L^2(\Omega)]^3 \rightarrow  [L^2(\Omega)]^3,
    \\
    \kappa_{x}(a) := \begin{pmatrix}-x_2 \\ x_1\end{pmatrix} a,
    &\quad &\quad&\quad & \kappa_{x}(a) :=  x \times a.
  \end{array}
\end{align*}
We note that the definition of the space $\Lambda_h^V$ is motivated by
the (Helmholtz-like) decomposition
\begin{equation}
  \label{polydecomp}
  \sbr[0]{P^{k-2}(\omega_V)}^{d} = \nabla P^{k-1}(\omega_V) \oplus \kappa_{x - V}([P_{k-3}(\omega_V)]^{d(d-1)/2}),
\end{equation}
where the first term of the right hand side is rotational-free and the
second term is divergence-free, see for example
\cite[(3.11)]{Arnold:2006}. Furthermore, we define the tensor space
\begin{equation*}
  \boldsymbol{Q}_h^V := \cbr[1]{Q^V_h \times \bar{Q}^V_h} \slash \mathbb{R}
  = \cbr[2]{ q_h,\bar{q}_h \in Q^V_h \times \bar{Q}^V_h:
    \sum_{K \in \mathcal{T}_V} \int_K q_h \dif x + \sum_{F \in \mathcal{F}_V}\int_F \bar{q}_h \dif s = 0 }.
\end{equation*}

\subsubsection{Definition and analysis of a local problem}

Let $G^V(\cdot) \in \del[1]{\boldsymbol{Q}^V_h}'$ be a given right
hand side (defined in \cref{sss:defanalEDGrecop}). We define the local
problem: find
$(\sigma_h^V, \boldsymbol{p}_h^V, \lambda_h^V) \in \Sigma_h^V \times
\boldsymbol{Q}_h^V \times \Lambda_h^V$ such that
\begin{subequations}
  \label{eq:local_prob}
  \begin{align}
    \label{eq:local_prob_mom}
    a^V_h(\sigma^V_h, \tau^V_h) + b^V_{1,h}(\boldsymbol{p}^V_h, \tau^V_h) + b^V_{2,h}(\lambda^V_h, \tau^V_h )&= 0
    && \forall \tau^V_h \in \Sigma^V_h,
    \\
    \label{eq:local_prob_mass}
    b^V_{1,h}(\boldsymbol{q}^V_h, \sigma^V_h) &= G^V(\boldsymbol{q}^V_h)
    && \forall \boldsymbol{q}^V_h \in \boldsymbol{Q}^V_h, \\
    \label{eq:local_prob_ortho}
    b^V_{2,h}(\mu^V_h, \sigma^V_h) &= 0
    && \forall \mu^V_h \in \Lambda^V_h.    
  \end{align}
\end{subequations}
where
\begin{subequations}
  \begin{align}
    \label{eq:formA_local}
    a^V_h(\sigma^V_h, \tau^V_h)
    :=& \sum_{K\in\mathcal{T}_V}\int_{K}\sigma_h^V \cdot \tau_h^V\dif x,
    \\
    \label{eq:formBone_local}
    b^V_{1,h}(\boldsymbol{q}^V_h, \sigma^V_h)
    :=& \sum_{K\in\mathcal{T}_V}\int_{K} \nabla \cdot \sigma_h^V q_h^V \dif x
        + \sum_{F \in \mathcal{F}_V}\int_{F} \jump{\sigma_h^V \cdot n } \bar{q}_h^V \dif s,
    \\
    \label{eq:formBtwo_local}
    b^V_{2,h}(\mu^V_h, \sigma^V_h)
    :=& \sum_{K\in\mathcal{T}_V}\int_{K} \sigma_h^V \cdot \mu_h^V \dif x.
  \end{align}
  \label{eq:bilin_forms_local}
\end{subequations}
For the stability analysis of the local problem \cref{eq:local_prob}
we define the following mesh dependent norms:
\begin{align*}
  \norm[0]{\tau_h^V}^2_{\Sigma_h^V}
  &:= \norm[0]{\tau_h^V}^2_{L^2(\omega_V)} + h^2 \norm[0]{\nabla \cdot \tau_h^V}^2_{L^2(\omega_V)}
    + h \norm[0]{\jump{\tau_h^V \cdot n}}^2_{L^2(\mathcal{F}_V)} && \forall \tau_h^V \in \Sigma_h^V,
  \\
  \norm[0]{\boldsymbol{q}_h^V}^2_{\boldsymbol{Q}_h^V}
  &:= \frac{1}{h^2} \norm[0]{ q_h^V }^2_{L^2(\omega_V)}
    + \frac{1}{h} \norm[0]{ \bar{q}_h^V }^2_{L^2(\mathcal{F}_V)}  &&\forall \boldsymbol{q}_h^V \in \boldsymbol{Q}_h^V,
  \\
  \norm[0]{\mu_h^V}_{\Lambda_h^V} &:=\norm[0]{ \mu_h^V }_{L^2(\omega_V)} && \forall \mu_h^V \in \Lambda_h^V.
\end{align*}

\begin{theorem}[Stability of \cref{eq:local_prob}]
  There exists a unique solution
  $(\sigma_h^V, \boldsymbol{p}_h^V, \lambda_h^V) \in \Sigma_h^V \times
  \boldsymbol{Q}_h^V \times \Lambda_h^V$ to problem
  \cref{eq:local_prob} with the stability estimate
  \begin{align}
    \label{eq:localstability}
    \norm[0]{\sigma_h^V}_{\Sigma_h^V} + \norm[0]{\boldsymbol{p}_h^V}_{\boldsymbol{Q}_h^V}
    + \norm[0]{\lambda_h^V}_{\Lambda_h^V} \lesssim \norm[0]{ G^V}_{(\boldsymbol{Q}_h^V)'}.
  \end{align}
\end{theorem}

\begin{proof}
  The proof is based on theory of mixed saddle point problems (see
  \cite[Chapter 4]{Boffi:book}); we need to prove kernel coercivity of
  the bilinear form $a_h^V(\cdot, \cdot)$ and the inf-sup condition of
  the constraints given by the bilinear forms
  $b_{1,h}^V(\cdot, \cdot)$ and $b_{2,h}^V(\cdot, \cdot)$.

  Let $\sigma_h^V \in \Sigma_h^V$ such that
  $b^V_{1,h}(\boldsymbol{q}^V_h, \sigma^V_h) + b^V_{2,h}(\mu^V_h,
  \sigma^V_h) = 0$ for all
  $ \boldsymbol{q}^V_h \in \boldsymbol{Q}^V_h$ and for all
  $\mu^V_h \in \Lambda^V_h.$ Next, note that
  $(-\nabla \cdot \sigma_h^V,\jump{\sigma_h^V \cdot n} ) \in
  \boldsymbol{Q}_h^V$ since integration by parts shows
  \begin{equation*}
    -\sum_{K \in \mathcal{T}_V} \int_K \nabla \cdot \sigma_h^V \dif x
    + \sum_{F \in \mathcal{F}_V} \int_F \jump{\sigma_h^V \cdot n} \dif s = 0.
  \end{equation*}
  Choosing
  $\boldsymbol{q}^V_h = (-\nabla \cdot \sigma_h^V,\jump{\sigma_h^V
    \cdot n} )$ and $\mu_h^V = 0$ in
  $b^V_{1,h}(\boldsymbol{q}^V_h, \sigma^V_h) + b^V_{2,h}(\mu^V_h,
  \sigma^V_h) = 0$ results in
  \begin{equation}
    \label{eq:kernelcoercivity_ah}
    \norm[0]{\sigma_h^V}_{\Sigma_h^V}^2 = \norm[0]{\sigma_h^V}_{L^2(\omega_V)}^2 = a_h^V(\sigma_h^V, \sigma_h^V),
  \end{equation}
  proving kernel coercivity of $a_h^V(\cdot, \cdot)$.

  We continue with the inf-sup condition of $b^V_{1,h}(\cdot,
  \cdot)$. To this end let $\boldsymbol{q}_h^V \in \boldsymbol{Q}_h^V$
  be arbitrary. Using \cite[Lemma 3 and Lemma 9]{Braess:2008} we find
  a function $\sigma^q_h \in \Sigma_h^V$ such that on each cell
  $K \in \mathcal{T}_V$ we have $-\nabla \cdot \sigma^q_h = q_h^V$, on
  all facets $F \in \mathcal{F}_V$ we have
  $\jump{\sigma^q_h\cdot n} = \bar{q}^V_h$, and
  $\sigma^q_h \in \Sigma_h^V$ satisfies the stability estimate
  $\norm[0]{\sigma^q_h}_{\Sigma_h^V} \lesssim
  \norm[0]{\boldsymbol{q}_h}_{\boldsymbol{Q}_h^V}$. Together these
  results prove the inf-sup condition
  \begin{equation}
    \label{eq:infsupone}
    \inf\limits_{\sigma_h^V \in \Sigma_h^V} \frac{ b^V_{1,h}(\boldsymbol{q}^V_h, \sigma^V_h) }{\norm[0]{\sigma^V_h}_{\Sigma_h^V}}
    \gtrsim \norm[0]{\boldsymbol{q}_h}_{\boldsymbol{Q}_h^V}.
  \end{equation}

  We next consider $b^V_{2,h}(\cdot, \cdot)$. Choose an arbitrary
  $\mu_h^V=\kappa_{x - V}(\zeta) \in \Lambda_h^V$ where
  $\zeta \in ([P_{k-3}(\omega_V)]^{d(d-1)/2})$ and set
  $\sigma^\mu_h = \operatorname{curl}(\phi_V \zeta)$, where $\phi_V$
  is the corresponding linear hat function of the fixed vertex
  $V$. Since $\phi_V \zeta$ is a polynomial of order $k-2$ and since
  $\phi_V$ vanishes at the boundary $\partial \omega_V$, so that
  $\sigma^\mu_h \cdot n = 0$ on $\partial \omega_V$, we have
  $\sigma^\mu_h \in \Sigma_h^V$. (Here $\operatorname{curl}$ is the
  usual curl operator in three dimensions. In two dimensions it is
  defined as the rotated gradient, i.e.
  $\operatorname{curl} = (-\partial_{x_2}, \partial_{x_1})$.) Now note
  that $b_{1,h}(\boldsymbol{q}^V_h, \sigma^\mu_h) = 0$ for all test
  functions $\boldsymbol{q}^V_h \in \boldsymbol{Q}^V_h$ since the
  divergence of a curl is zero on $K \in \mathcal{T}_V$ and the normal
  jump disappears by properties of the discrete de Rham complex (see,
  for example, \cite{Boffi:book}). With the same steps as in the proof
  of \cite[Theorem 12]{Lederer:2017b} we obtain the inf-sup condition
  (on the kernel of $b_{1,h}(\cdot, \cdot)$)
  \begin{equation*}
    \inf\limits_{\substack{\sigma_h^V \in \Sigma_h^V \\ b_{1,h}(\boldsymbol{q}^V_h, \sigma^\mu_h) = 0, \forall \boldsymbol{q}^V_h \in \boldsymbol{Q}^V_h}}
    \frac{ b^V_{2,h}(\mu^V_h, \sigma^\mu_h)}{\norm[0]{\sigma^V_h}_{\Sigma_h^V}} \gtrsim \norm[0]{\mu^V_h}_{\Lambda_h^V}.
  \end{equation*}
  The coupled inf-sup condition for $b_{1,h} + b_{2,h}$ now follows by
  application of \cite[Theorem 3.1]{Howell:2011}. Together with the
  kernel coercivity \cref{eq:kernelcoercivity_ah} we conclude
  existence, uniqueness and stability of \cref{eq:local_prob}.
\end{proof}

\subsubsection{Definition and analysis of the EDG-reconstruction operator}
\label{sss:defanalEDGrecop}

In this section we define the EDG-reconstruction operator. To this end
we first introduce an operator defined on vertex patches.

Let $V \in \mathcal{V}$ be an arbitrary vertex and let
$F \in \mathcal{F}_V$ be an arbitrary edge on the corresponding vertex
patch of $V$. Furthermore, let $\bar{q}_h \in P^{m}(F)$ with
$m = k,k-1$, let $\varphi_i$ with $i = 0,\ldots,m$ be a Lagrangian
basis of the polynomial space $P^{m}(F)$, and let $\xi_i$ be the
associated Lagrangian points. Then for $x \in F$ we can write
$\bar{q}_h(x) = \sum_{i=0}^{m} c_i \varphi_i(x)$ where $c_i$ are the
related coefficients. We define the bubble projection
$\mathbb{B}_F^{V}: P^{m}(F) \rightarrow P^{m}(F)$ as
\begin{equation*}
  \mathbb{B}_F^{V}(\bar{q}_h)(x) = \sum_{i=0}^{m} c_i\phi_V(\xi_i) \varphi_i(x).
\end{equation*}
The definition of the bubble projector $\mathbb{B}_F^{V}$ reads as a
simple weighting of the coefficients $c_i$ with the value of the hat
function at the related Lagrangian point $\xi_i$. For the case $d=2$,
let $V_o$ be the opposite vertex of $V$ of $F$. By definition,
$\mathbb{B}_F^{V}(\bar{q}_h)(V_o) = 0$ and we have the identity
\begin{equation}
  \label{bubbleid}
  \mathbb{B}_F^{V}(\bar{q}_h)(x) +  \mathbb{B}_F^{V_o}(\bar{q}_h)(x)
  =  \sum_{i=0}^{m} c_i(\phi_V(\xi_i) + \phi_{V_o}(\xi_i)) \varphi_i(x)
  = \bar{q}_h(x).
\end{equation}
For the case $d = 3$, let $V_{o_1}$ and $V_{o_2}$ be the opposite
vertices of $V$ of the triangle $F$, and let $E_{12}$ be the edge
connecting the vertices $V_{o_1}$ and $V_{o_2}$. With the same
arguments as above we have $\mathbb{B}_F^{V}(\bar{q}_h)|_{E_{12}} = 0$
and
\begin{equation}
  \label{bubbleid_three}
  \mathbb{B}_F^{V}(\bar{q}_h)(x) +  \mathbb{B}_F^{V_{o_1}}(\bar{q}_h)(x) + \mathbb{B}_F^{V_{o_2}}(\bar{q}_h)(x) =\bar{q}_h(x).
\end{equation}
Note that the definition for $d=3$ equals the bubble projection on
triangles defined in \cite[Section 4.1]{Lederer:2017b}.

Now, let $u_h \in V_h$ be an arbitrary but fixed discrete velocity,
and let $\sigma_h^V$ be the solution to the local problem
\cref{eq:local_prob} defined on vertex patches with the following
right hand side:
\begin{equation}
  \label{eq:localrhs}
  G^V(\boldsymbol{q}_h)
  = \sum_{F \in \mathcal{F}_V} \int_F \jump{u_h \cdot n} \mathbb{B}_F^V\circ (\operatorname{id} - \Pi^0_{\mathcal{F}_V})(\bar{q}_h) \dif s
  \quad \forall \boldsymbol{q}_h = (q_h,\bar{q}_h) \in \boldsymbol{Q}_h,
\end{equation}
where $\Pi^0_{\mathcal{F}_V}$ is the projection onto constants on
$\mathcal{F}_V$, i.e.,
\begin{equation*}
  \Pi^0_{\mathcal{F}_V}(q) = \frac{1}{|\mathcal{F}_V|} \sum_{F \in \mathcal{F}_V} \int_F q \dif s
  \quad \forall q \in L^2(\mathcal{F}_V).
\end{equation*}

\begin{theorem}[Properties of the local solutions]
  \label{thm:prop_loc_sol}
  Let $u_h \in V_h$, let $V \in \mathcal{V}$ be arbitrary but fixed,
  and let $\sigma_h^V$ be the solution to \cref{eq:local_prob} with
  the right hand side defined by \cref{eq:localrhs}. Furthermore, let
  $\sigma_h^V$ be trivially extended to $\Omega \setminus \omega_V$ by
  zero. 
  \begin{enumerate}[label=(\roman*.)]
  \item For all  $(q_h,\bar{q}_h^d) \in Q_h \times \bar{Q}_h^d$ we have
    \begin{equation*}
      (\nabla \cdot \sigma_h^V, q_h)_{L^2(\mathcal{T})}
      + (\jump{\sigma_h^V \cdot n}, \bar{q}_h^d)_{L^2(\mathcal{F})}
      = (\mathbb{B}_F^V \circ (\operatorname{id} - \Pi^0_{\mathcal{F}_V})(\jump{u_h\cdot n}) , \bar{q}_h^d)_{L^2(\mathcal{F})}.
    \end{equation*}
    \label{prop_localsol_one}
  \item We have
    $ \norm[0]{\sigma_h^V}_{\Sigma_h^V} \lesssim \sqrt{h} \norm[0]{
      \jump{u_h \cdot n}}_{L^2(\mathcal{F}_V)}$.
    \label{prop_localsol_two}
  \item If $b_h(\boldsymbol{q}_h, u_h) = 0$ for all
    $\boldsymbol{q}_h \in \boldsymbol{Q}_h$ then
    $(\sigma_h^V, \eta_h)_{L^2(\omega_V)} = 0$ for all
    $\eta_h \in P^{k-2}(\omega_V)$.
    \label{prop_localsol_three}
  \end{enumerate}
\end{theorem}

\begin{proof}
  We start with the proof of \cref{prop_localsol_one}. We first show
  that the solution $\sigma_h^V$ satisfies
  \begin{equation*}
    b_{1,h}^V((c,c), \sigma_h^V) = G^V( (c,c)) \quad \forall c \in \mathbb{R}.
  \end{equation*}
  To show this we first note that the left hand side vanishes by
  integration by parts, i.e., $ b_{1,h}((c,c), \sigma_h^V)=0$. Next,
  note that $\Pi^0_{\mathcal{F}_V}(c) = c$ and so the right hand side also vanishes.

  Now let $(q_h,\bar{q}_h^d) \in Q_h \times \bar{Q}_h^d$ be arbitrary
  and set
  $\boldsymbol{q}_h^V = (q_h|_{\omega_V}, \bar{q}_h^d|_{\omega_V})$,
  i.e., the restriction on the vertex patch. We define the constant
  \begin{equation*}
    c := \frac{1}{|\mathcal{T}_V| + |\mathcal{F}_V|} \del{\sum_{K \in \mathcal{T}_V} \int_K q_h|_{\omega_V} \dif x
      + \sum_{F \in \mathcal{F}_V} \int_F \bar{q}_h^d|_{\omega_V} \dif s}, 
  \end{equation*}
  and write
  $\boldsymbol{q}_h^V = (q_h|_{\omega_V}- c ,\bar{q}_h^d|_{\omega_V} -
  c) + (c,c)$. The first term of the right hand side is an element of
  $\boldsymbol{Q}_h^V$, and so with the above findings and
  \cref{eq:local_prob_mass} we observe that
  \begin{equation*}
    b_{1,h}^V(\boldsymbol{q}_h^V, \sigma_h^V) = G^V(\boldsymbol{q}_h^V).
  \end{equation*}
  Using that $\sigma_h^V$ is zero on $\Omega \setminus \omega_V$ we
  then find that
  \begin{align*}
    (\nabla \cdot \sigma_h^V, q_h)_{L^2(\mathcal{T})} + (\jump{\sigma_h^V \cdot n}, \bar{q}_h)_{L^2(\mathcal{F})}
    &= b_{1,h}^V(\boldsymbol{q}_h^V, \sigma_h^V) = G^V(\boldsymbol{q}_h^V)
    \\
    &= ( \mathbb{B}_F^V\circ (\operatorname{id} - \Pi^0_{\mathcal{F}_V})(\jump{u_h\cdot n}), \bar{q}_h)_{L^2(\mathcal{F}_V)}.
  \end{align*}
  
  We next prove \cref{prop_localsol_two}. Using the stability result
  \cref{eq:localstability} this follows by definition of the dual
  norm, 
  \begin{align*}
    \norm[0]{\sigma_h^V}_{\Sigma_h^V} \lesssim \norm[0]{G^V}_{(\boldsymbol{Q}_h^V)'}
    &:=\sup\limits_{0\ne \boldsymbol{r}_h^V\in\boldsymbol{Q}_h^V}\frac{G^V(\boldsymbol{r}_h^V)}{\norm[0]{\boldsymbol{r}_h^V}_{\boldsymbol{Q}_h^V}}
    \\
    & \le
      \sup\limits_{0\ne \boldsymbol{r}_h^V\in\boldsymbol{Q}_h^V}
      \frac{\norm[0]{ \mathbb{B}_F^V\circ(\operatorname{id}-\Pi_{\mathcal{F}_V}^0)\jump{u_h \cdot n}}_{L^2(\mathcal{F}_V)}
      \norm[0]{\bar{r}_h^V}_{L^2(\mathcal{F}_V)}}{\norm[0]{\boldsymbol{r}_h^V}_{\boldsymbol{Q}_h^V}}
    \\    
    &\le \sqrt{h} \norm[0]{ \mathbb{B}_F^V\circ(\operatorname{id}-\Pi_{\mathcal{F}_V}^0)\jump{u_h \cdot n}}_{L^2(\mathcal{F}_V)}
    \\
    &\le \sqrt{h} \norm[0]{\jump{u_h \cdot n} }_{L^2(\mathcal{F}_V)}
      + \sqrt{h} \norm[0]{\Pi_{\mathcal{F}_V}^0\jump{u_h \cdot n} }_{L^2(\mathcal{F}_V)}
    \\
    &\lesssim \sqrt{h} \norm[0]{\jump{u_h \cdot n}}_{L^2(\mathcal{F}_V)},
  \end{align*}
  where we used continuity of $\Pi_{\mathcal{F}_V}^0$ and the bubble
  projector since the weighting of the coefficient lies in $[0,1]$.
  
  Finally, we prove \cref{prop_localsol_three}. Let
  $\eta_h \in P^{k-2}(\omega_V)$ be arbitrary. Using decomposition
  \cref{polydecomp} we find polynomials
  $\theta \in P^{k-1}(\omega_V)$, and
  $\zeta \in ([P_{k-3}(\omega_V)]^{d(d-1)/2})$ such that
  \begin{equation*}
    \eta_h = \nabla \theta +  \kappa_{x - V}(\zeta).
  \end{equation*}
  Note that $\theta$ can be chosen such that
  $(\theta, \theta) \in \boldsymbol{Q}_h^V$ since the above
  decomposition includes only the gradient of $\theta$. Using
  \cref{eq:local_prob_ortho}, we note that
  $(\sigma_h^V, \kappa_{x - V}(\zeta))_{L^2(\omega_V)} = 0$. Next,
  using integration by parts,
  \begin{align*}
    (\sigma_h^V, \nabla \theta)_{\omega_V}
    &= b_{1,h}^V((\theta,\theta), \sigma_h^V) = G^V((\theta,\theta))
    \\
    &= \sum_{F \in \mathcal{F}_V} \int_F \jump{u_h \cdot n} \mathbb{B}_F^V \circ (\operatorname{id} - \Pi^0_{\mathcal{F}_V}) \theta \dif s = 0,
  \end{align*}
  since
  $\mathbb{B}_F^V \circ (\operatorname{id} - \Pi^0_{\mathcal{F}_V})
  \theta$ is an element of $\bar{Q}_h$.
\end{proof}

We are now able to define a reconstruction operator that satisfies the
properties of \cref{as:R_assumption}.

\begin{lemma}
  \label{lem:def:reconstruction}
  Let $u_h \in V_h$ and let $\sigma_h^V$ be the solution to
  \cref{eq:local_prob} with the right hand side defined by
  \cref{eq:localrhs} for an arbitrary vertex $V \in \mathcal{V}$. The
  reconstruction operator $\mathcal{R}: V_h \rightarrow V_h$ defined
  by
  \begin{equation*}
    \mathcal{R}(u_h) := u_h - \sigma_h \quad \textrm{with} \quad  \sigma_h := \sum_{V \in \mathcal{V}} \sigma_h^V.  
  \end{equation*}
  satisfies the properties of \cref{as:R_assumption}.
\end{lemma}

\begin{proof}
  We first prove that the reconstruction operator satisfies
  \cref{as:R_assumption_one} of \cref{as:R_assumption}. Let
  $u_h \in V_h$ such that $b_h(\boldsymbol{r}_h, u_h) = 0$ for all
  $\boldsymbol{r}_h \in \boldsymbol{Q}_h$. Now let
  $(q_h, \bar{q}_h^d) \in Q_h \times \bar{Q}_h^d$ be arbitrary, then
  \begin{equation}
    \label{eq:decompRuh}
    \begin{split}
      \sum_{K \in \mathcal{T}} \int_K \nabla \cdot \mathcal{R}(u_h) & q_h \dif x
      + \sum_{F \in \mathcal{F}} \int_F \jump{\mathcal{R}(u_h) \cdot n} \bar{q}_h^d \dif s
      \\
      =& \sum_{K \in \mathcal{T}} \int_K \nabla \cdot u_h q_h \dif x
      + \sum_{F \in \mathcal{F}} \int_F \jump{u_h \cdot n} \bar{q}_h^d \dif s
      \\
      & - \sum_{K \in \mathcal{T}} \int_K \nabla \cdot \sigma_h q_h \dif x
      - \sum_{F \in \mathcal{F}} \int_F \jump{\sigma_h \cdot n} \bar{q}_h^d \dif s.      
    \end{split}
  \end{equation}
  Consider the last two terms on the right hand side. By definition of
  $\sigma_h$,
  \begin{align*}
    - \sum_{K \in \mathcal{T}}
    & \int_K \nabla \cdot \sigma_h q_h \dif x
      - \sum_{F \in \mathcal{F}} \int_F \jump{\sigma_h \cdot n} \bar{q}_h^d \dif s
    \\
    &= - \sum_{K \in \mathcal{T}} \sum_{V \in K} \int_K \nabla \cdot \sigma^V_h q_h \dif x
      - \sum_{F \in \mathcal{F}} \sum_{V \in F} \int_F \jump{\sigma^V_h \cdot n} \bar{q}_h^d \dif s
    \\
    &= \sum_{V \in \mathcal{V}}\del[2]{ - \sum_{K \in \mathcal{T}_V} \int_K \nabla \cdot \sigma^V_h q_h \dif x
      - \sum_{F \in \mathcal{F}_V} \int_F \jump{\sigma^V_h \cdot n} \bar{q}_h^d \dif s }
    \\
    &=  \sum_{V \in \mathcal{V}}\del[2]{- \sum_{F \in \mathcal{F}_V}
      \int_F \mathbb{B}_F^V \circ (\operatorname{id} - \Pi^0_{\mathcal{F}_V})(\jump{u_h\cdot n} ) \bar{q}_h^d  \dif s  }
    \\
    &=  \sum_{V \in \mathcal{V}}\del[2]{- \sum_{F \in \mathcal{F}_V} \int_F \mathbb{B}_F^V (\jump{u_h\cdot n}) \bar{q}_h^d  \dif s
      + \sum_{F \in \mathcal{F}_V} \int_F  \jump{u_h\cdot n} \mathbb{B}_F^V \circ \Pi^0_{\mathcal{F}_V}(\bar{q}_h^d)  \dif s  },    
  \end{align*}
  where we used \cref{prop_localsol_one} of \cref{thm:prop_loc_sol}
  for the third equality, and that
  $\mathbb{B}_F^V \circ \Pi^0_{\mathcal{F}_V}$ is a symmetric self
  adjoint operator for the fourth equality. Since $u_h$ is discretely
  divergence-free and
  $\mathbb{B}_F^V \circ \Pi^0_{\mathcal{F}_V}(\bar{q}_h) \in
  \bar{Q}_h$ the second sum on the right hand side
  vanishes. Therefore, using \cref{bubbleid_three},
  \begin{align}
    \label{eq:divsigmahdef}
    -\sum_{K \in \mathcal{T}} \int_K & \nabla \cdot \sigma_h q_h \dif x
      - \sum_{F \in \mathcal{F}} \int_F \jump{\sigma_h \cdot n} \bar{q}_h^d \dif s
    \\ \nonumber
    & =  \sum_{V \in \mathcal{V}}  \sum_{F \in \mathcal{F}_V} - \int_F  \mathbb{B}_F^V (\jump{u_h\cdot n}) \bar{q}_h^d  \dif s
    \\ \nonumber
    & =  \sum_{F \in \mathcal{F}} \sum_{V \in F} - \int_F \mathbb{B}_F^V(\jump{u_h\cdot n}) \bar{q}_h^d \dif s
      =  -\sum_{F \in \mathcal{F}} \int_F \jump{u_h\cdot n} \bar{q}_h^d  \dif s.
  \end{align}
  Combining \cref{eq:decompRuh} and \cref{eq:divsigmahdef} we find
  that
  \begin{multline*}
    \sum_{K \in \mathcal{T}} \int_K \nabla \cdot \mathcal{R}(u_h) q_h \dif x
    + \sum_{F \in \mathcal{F}} \int_F \jump{\mathcal{R}(u_h) \cdot n} \bar{q}_h^d \dif s
    \\
    = \sum_{K \in \mathcal{T}} \int_K \nabla \cdot u_h q_h \dif x
    + \sum_{F \in \mathcal{F}} \int_F \jump{u_h \cdot n} \bar{q}_h^d \dif s 
    -  \sum_{F \in \mathcal{F}} \int_F \jump{u_h\cdot n} \bar{q}_h^d \dif s = 0,
  \end{multline*}
  where we used that $\nabla \cdot u_h$ vanishes in each cell. Now,
  since $\nabla \cdot \mathcal{R}(u_h) \in Q_h$ and
  $\jump{\mathcal{R}(u_h) \cdot n} \in \bar{Q}_h^d$ we conclude the
  proof.

  To prove that the reconstruction operator satisfies
  \cref{as:R_assumption_two} of \cref{as:R_assumption} we use
  \cref{prop_localsol_two} from \cref{thm:prop_loc_sol}:
  \begin{align*}
    \norm[0]{\mathcal{R}(u_h) - u_h }^2_{L^2(\Omega)}
    & = \norm[0]{\sigma_h}^2_{L^2(\Omega)}
      \le \sum_{V \in \mathcal{V}} \norm[0]{\sigma_h^V}^2_{L^2(\omega_V)}
    \\
    & \le \sum_{V \in \mathcal{V}} \norm[0]{\sigma_h^V}^2_{\Sigma_h^V}
      \lesssim \sum_{V \in \mathcal{V}} h \norm[0]{\jump{u_h \cdot n}}^2_{L^2(\mathcal{F}_V)}
      \le h^2 \tnorm{u_h}^2_1.
  \end{align*}
  The result follows after taking the square root on both sides.

  Finally, we prove \cref{as:R_assumption_three} of
  \cref{as:R_assumption}. Let $g \in [L^2(\Omega)]^d$ be
  arbitrary. Using \cref{prop_localsol_three} and
  \cref{prop_localsol_two} from \cref{thm:prop_loc_sol} we find that
  \begin{align*}
    (g, \mathcal{R}(u_h) - u_h)_{L^2(\Omega)}
    &= \sum_{V \in \mathcal{V}} \int_{\omega_V} \sigma_h^V \cdot g \dif x
    \\
    &= \sum_{V \in \mathcal{V}} \int_{\omega_V} \sigma_h^V \cdot (g - \Pi^{k-2}_{\omega_V}g) \dif x
    \\
    &\le \sum_{V \in \mathcal{V}} \sbr{ \del[2]{\frac{1}{h}\norm[0]{\sigma_h^V}_{L^2(\omega_V)}}
      \del[2]{h^2\int_{\omega_V}(g - \Pi^{k-2}_{\omega_V}g)^2 \dif x}^{1/2}}
    \\
    &\le \del[2]{\sum_{V \in \mathcal{V}} \frac{1}{h^2}\norm[0]{\sigma_h^V}^2_{L^2(\omega_V)}}^{1/2}
      \del[2]{\sum_{V \in \mathcal{V}} h^2\int_{\omega_V}(g - \Pi^{k-2}_{\omega_V}g)^2 \dif x}^{1/2}
    \\
    &\le \del[2]{\sum_{V \in \mathcal{V}} \frac{1}{h} \norm[0]{\jump{u_h \cdot n}}^2_{\mathcal{F}_V}}^{1/2}
      \del[2]{\sum_{V \in \mathcal{V}} h^2\int_{\omega_V}(g - \Pi^{k-2}_{\omega_V}g)^2 \dif x}^{1/2}
    \\
    & \le \tnorm{u_h}_1 \del[2]{\sum_{V \in \mathcal{V}} h^2\int_{\omega_V}(g - \Pi^{k-2}_{\omega_V}g)^2 \dif x}^{1/2},
  \end{align*}
  concluding the proof.
\end{proof}

We remark that the reconstruction operator defined in
\cref{lem:def:reconstruction} can easily be implemented in existing
EDG codes for the Stokes problem since it is applied only to the right
hand side of the discretization \cref{eq:discrete_stokes_pr}. The left
hand side matrices of \cref{eq:discrete_stokes} and
\cref{eq:discrete_stokes_pr} are \emph{identical}.

\section{Numerical tests}
\label{sec:numerical}

In this section we present two and three dimensional numerical
examples that demonstrate that the modified EDG discretization
\cref{eq:discrete_stokes_pr} with the reconstruction operator defined
in \cref{lem:def:reconstruction} results in a pressure-robust
discretization of the Stokes equations. All numerical examples have
been implemented in the higher-order finite element library Netgen/NGSolve
\cite{Schoberl:1997, Schoberl:2014}.

We study the Stokes problem~\cref{eq:stokes} on $\Omega :=
[0,1]^d$. For $d=2$ we set the body force $f$ such that the exact
solution is given by
\begin{equation*}
  u = \operatorname{curl}(\xi)\quad \textrm{ and } \quad p = x_1^5+x_2^5-1/3,
\end{equation*}
where $\xi = x_1^2(x_1-1)^2x_2^2(x_2-1)^2$, while for $d=3$ the body
force is such that the exact solution is given by
\begin{equation*}
  u = \operatorname{curl}((\xi,\xi,\xi))\quad \textrm{ and } \quad p = x_1^5+x_2^5+x_3^5-1/2,
\end{equation*}
where $\xi = x_1^2(x_1-1)^2x_2^2(x_2-1)^2x_3^2(x_3-1)^2$.

In this section we denote the discrete velocity solution obtained by
the EDG method \cref{eq:discrete_stokes} by $u_h$ and the discrete
velocity solution obtained by the modified EDG method
\cref{eq:discrete_stokes_pr} by $u_h^{\star}$. We furthermore define
$e_h = u - u_h$ and $e_h^{\star} = u - u_h^{\star}$.
  
For the two dimensional test case we plot, in \cref{fig:2dL2norm} and
\cref{fig:2dH1seminorm}, respectively, the $L^2$-norm and
$H^1$-seminorm (the $L^2$-norm of the gradient) of $e_h$ and
$e_h^{\star}$. We compute the solution for polynomial orders $k=2,3,4$
and for both $m=k$ and $m=k-1$ in \cref{eq:approx_spaces_edg}. We fix
the viscosity to $\nu = 10^{-6}$.

From both figures we first observe optimal rates of convergence for
all methods. We also observe that $u_h^{\star}$ is not affected by the
choice of $m$. However, the choice of $m$ does affect $u_h$, as we
discuss next.

For the case $m=k-1$ we observe that $u_h^{\star}$ is significantly
more accurate than $u_h$ when $\nu=10^{-6}$; $e_h^{\star}$ is
$10^3$--$10^5$ times smaller (depending on $k$) than $e_h$ in both the
$L^2$-norm and in the $H^1$-seminorm.

When $m=k$ we note that $e_h$ is approximately $10^2$ times smaller
than when $m=k-1$. We conjecture that this is due to a better
enforcement of continuity of the normal component of the velocity. As
such, when $m=k$, $u_h$ is `closer' to an $H(\divergence)$-conforming
velocity than when $m=k-1$ mitigating the role of the pressure-error
in \cref{eq:notpressurerob}. Due to $u_h$ being more accurate when
$m=k$, and that the accuracy of $u_h^{\star}$ does not seem to depend
on $m$, we observe that $u_h^{\star}$ is `only' $10$--$10^2$ times
more accurate (depending on $k$) than $u_h$.

We now vary the viscosity from $\nu = 1$ to $\nu=10^{-9}$ and plot the
$H^1$-seminorm on a fixed mesh with $|\mathcal{T}|=230$ triangles for
polynomial orders $k=2,3,4$ in \cref{fig:2dnuH1seminorm}. We observe
that $e_h^{\star}$ is not affected by $\nu$ thereby verifying
\cref{cor:ratesofconv}. As expected from \cref{eq:notpressurerob}, the
accuracy of $u_h$ deteriorates as viscosity decreases. Furthermore, in
agreement with our previous observations, the solution $u_h$ with
$m=k$ is approximately $10^2$ times more accurate than when $m=k-1$.

In \cref{fig:3derrors} we plot the $L^2$-norm and $H^1$-seminorm of
the error of the discrete velocity for the three dimensional test
case. We again set $\nu = 10^{-6}$ and compute the solution for
polynomial orders $k=1,2$. We consider only the case $m=k$ since the
EDG method with $m=k-1$ is not defined for $k=1$. We draw the same
conclusions as in the two dimensional test case, namely optimal rates
of convergence for all methods and a pressure-robust discrete velocity
approximation when using the modified EDG method.

\pgfplotstableread{3dconvergence_nu_1e-06_order_1_phatlo_False_rec_True.dat}  \tdoneplfrect
\pgfplotstableread{3dconvergence_nu_1e-06_order_1_phatlo_False_rec_False.dat} \tdoneplfrecf
\pgfplotstableread{3dconvergence_nu_1e-06_order_2_phatlo_False_rec_True.dat}  \tdtwoplfrect
\pgfplotstableread{3dconvergence_nu_1e-06_order_2_phatlo_False_rec_False.dat} \tdtwoplfrecf

\pgfplotstableread{convergence_nu_1e-06_order_2_phatlo_False_rec_True.dat}  \twoplfrect
\pgfplotstableread{convergence_nu_1e-06_order_2_phatlo_False_rec_False.dat} \twoplfrecf
\pgfplotstableread{convergence_nu_1e-06_order_2_phatlo_True_rec_True.dat}   \twopltrect
\pgfplotstableread{convergence_nu_1e-06_order_2_phatlo_True_rec_False.dat}  \twopltrecf

\pgfplotstableread{convergence_nu_1e-06_order_3_phatlo_False_rec_True.dat}  \threeplfrect
\pgfplotstableread{convergence_nu_1e-06_order_3_phatlo_False_rec_False.dat} \threeplfrecf
\pgfplotstableread{convergence_nu_1e-06_order_3_phatlo_True_rec_True.dat}   \threepltrect
\pgfplotstableread{convergence_nu_1e-06_order_3_phatlo_True_rec_False.dat}  \threepltrecf

\pgfplotstableread{convergence_nu_1e-06_order_4_phatlo_False_rec_True.dat}  \fourplfrect
\pgfplotstableread{convergence_nu_1e-06_order_4_phatlo_False_rec_False.dat} \fourplfrecf
\pgfplotstableread{convergence_nu_1e-06_order_4_phatlo_True_rec_True.dat}   \fourpltrect
\pgfplotstableread{convergence_nu_1e-06_order_4_phatlo_True_rec_False.dat}  \fourpltrecf

\definecolor{myblue}{RGB}{62,146,255}
\definecolor{mygreen}{RGB}{22,135,118}
\definecolor{myred}{RGB}{255,145,0}

\begin{figure}  
\begin{tikzpicture}
  [
  scale=0.96
  ]
  \begin{axis}[
    name=plot1,
    scale=0.6,
    title ={ $m = k$},
    legend style={text height=0.7em },
    legend style={draw=none},
    style={column sep=0.1cm},
    ylabel=$L^2$-norm error,
    xlabel=number of elements $|\mathcal{T}|$,
    xmode=log,
    xmin = 1e1,
    xmax = 3e4,
    ymode=log,
    ytick = {1e-8, 1e-6,1e-4,1e-2,1, 1e2},
    y tick label style={
      /pgf/number format/.cd,
      fixed,
      precision=2
    },
    x tick label style={
      /pgf/number format/.cd,
      fixed,
      precision=2
    },
    %
    grid=both,
    legend style={
      cells={align=left},
      anchor = north west
    },
    ]

    \addlegendimage{only marks, mark=*,black}
    \addlegendimage{only marks, mark=triangle*,black}
    \addlegendimage{only marks, mark=square*,black}
    \addlegendimage{line width=0.5pt, myblue}
    \addlegendimage{line width=0.5pt, mygreen}
    \addlegendimage{dashed,line width=0.5pt}    
    \addlegendimage{densely dotted,line width=0.5pt}

    \addplot[line width=0.5pt, color=mygreen, mark=*] table[x=1, y=3]{\threeplfrect};
    \addplot[line width=0.5pt, color=mygreen, mark=triangle*] table[x=1, y=3]{\fourplfrect};
    \addplot[line width=0.5pt, color=mygreen, mark=square*] table[x=1, y=3]{\twoplfrect};

    \addplot[line width=0.5pt, color=myblue, mark=*] table[x=1, y=3]{\threeplfrecf};
    \addplot[line width=0.5pt, color=myblue, mark=triangle*] table[x=1, y=3]{\fourplfrecf};
    \addplot[line width=0.5pt, color=myblue, mark=square*] table[x=1, y=3]{\twoplfrecf};

        \addplot[line width=0.5pt, dashed, color=black] table[x=1,y expr={1/\thisrowno{1}^1.5/12}]{\threeplfrect};
    \addplot[line width=0.5pt, dashdotted, color=black] table[x=1,y expr={1/\thisrowno{1}^2/13}]{\threeplfrect};
    \addplot[line width=0.5pt, densely dotted, color=black] table[x=1,y expr={1/\thisrowno{1}^2.5/14}]{\threeplfrect};
    
  \end{axis}
   \begin{axis}[
     name=plot2,
     at=(plot1.right of south east),
     anchor=left of south west,
     title = {$m = k-1$},
     legend entries={$k=2$, $k=3$,$k=4$, $e_h$, $e_h^{\star}$, $h^3$, $h^4$,$h^5$},
     scale=0.6,
    legend style={text height=0.7em },
    legend style={draw=none},
    style={column sep=0.15cm},
    xlabel=number of elements $|\mathcal{T}|$,
    xmode=log,
    ymode=log,
    ytick = {1e-8,1e-6,1e-4,1e-2,1, 1e2, 1e4},
    y tick label style={
      /pgf/number format/.cd,
      fixed,
      precision=2
    },
    x tick label style={
      /pgf/number format/.cd,
      fixed,
      precision=2
    },
    %
    xmin = 1e1,
    xmax = 3e4,
    grid=both,
    legend style={
      cells={align=left},
      at={(1.05,1)},
      anchor = north west
    },
    ]

        \addlegendimage{only marks, mark=square*,black}
     \addlegendimage{only marks, mark=*,black}
    \addlegendimage{only marks, mark=triangle*,black}

    \addlegendimage{line width=0.5pt, myblue}
    \addlegendimage{line width=0.5pt, mygreen}
    \addlegendimage{dashed,line width=0.5pt}    
    \addlegendimage{dashdotted,line width=0.5pt}
    \addlegendimage{densely dotted,line width=0.5pt}

     \addplot[line width=0.5pt, color=mygreen, mark=*] table[x=1, y=3]{\threepltrect};
    \addplot[line width=0.5pt, color=mygreen, mark=triangle*] table[x=1, y=3]{\fourpltrect};
    \addplot[line width=0.5pt, color=mygreen, mark=square*] table[x=1, y=3]{\twopltrect};

    \addplot[line width=0.5pt, color=myblue, mark=*] table[x=1, y=3]{\threepltrecf};
    \addplot[line width=0.5pt, color=myblue, mark=triangle*] table[x=1, y=3]{\fourpltrecf};
    \addplot[line width=0.5pt, color=myblue, mark=square*] table[x=1, y=3]{\twopltrecf};

           \addplot[line width=0.5pt, dashed, color=black] table[x=1,y expr={1/\thisrowno{1}^1.5/12}]{\threeplfrect};
    \addplot[line width=0.5pt, dashdotted, color=black] table[x=1,y expr={1/\thisrowno{1}^2/13}]{\threeplfrect};
    \addplot[line width=0.5pt, densely dotted, color=black] table[x=1,y expr={1/\thisrowno{1}^2.5/14}]{\threeplfrect};
  \end{axis}
\end{tikzpicture}
\caption{Two dimensional test case as described in
  \cref{sec:numerical} using $\nu = 10^{-6}$. We plot the $L^2$-norm
  error of the velocity, against the number of elements in the mesh,
  for polynomial degrees $k=2,3,4$. Here $e_h = u-u_h$ and
  $e_h^{\star}=u-u_h^{\star}$ with $u_h$ the discrete velocity
  solution to the EDG method \cref{eq:discrete_stokes} and
  $u_h^{\star}$ the discrete velocity solution to the modified EDG
  method \cref{eq:discrete_stokes_pr}. The solutions are computed both
  for $m=k$ in \cref{eq:approx_spaces_edg} (left) and when $m=k-1$
  (right).}
\label{fig:2dL2norm}
\end{figure}
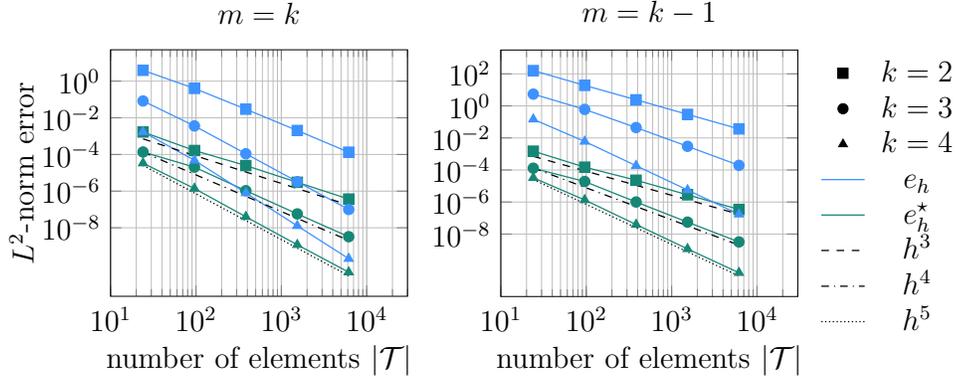

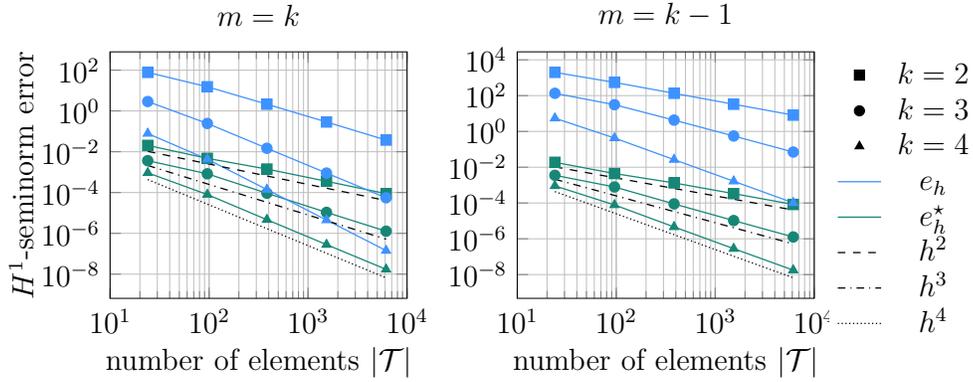
\begin{figure}  
\begin{tikzpicture}
  [
  scale=0.96
  ]
  \begin{axis}[
    name=plot1,
    scale=0.6,
    title ={ $m = k$},
    legend style={text height=0.7em },
    legend style={draw=none},
    style={column sep=0.1cm},
    ylabel=$H^1$-seminorm error,
    xlabel=number of elements $|\mathcal{T}|$,
    xmode=log,
    xmin = 1e1,
    xmax = 1e4,
    ymode=log,
    ytick = {1e-8, 1e-6,1e-4,1e-2,1, 1e2},
    y tick label style={
      /pgf/number format/.cd,
      fixed,
      precision=2
    },
    x tick label style={
      /pgf/number format/.cd,
      fixed,
      precision=2
    },
    %
    grid=both,
    legend style={
      cells={align=left},
      anchor = north west
    },
    ]

    \addlegendimage{only marks, mark=*,black}
    \addlegendimage{only marks, mark=triangle*,black}
    \addlegendimage{only marks, mark=square*,black}
    \addlegendimage{line width=0.5pt, myblue}
    \addlegendimage{line width=0.5pt, mygreen}
    \addlegendimage{dashed,line width=0.5pt}    
    \addlegendimage{densely dotted,line width=0.5pt}

    \addplot[line width=0.5pt, color=mygreen, mark=*] table[x=1, y=2]{\threeplfrect};
    \addplot[line width=0.5pt, color=mygreen, mark=triangle*] table[x=1, y=2]{\fourplfrect};
    \addplot[line width=0.5pt, color=mygreen, mark=square*] table[x=1, y=2]{\twoplfrect};

    \addplot[line width=0.5pt, color=myblue, mark=*] table[x=1, y=2]{\threeplfrecf};
    \addplot[line width=0.5pt, color=myblue, mark=triangle*] table[x=1, y=2]{\fourplfrecf};
    \addplot[line width=0.5pt, color=myblue, mark=square*] table[x=1, y=2]{\twoplfrecf};
    
    \addplot[line width=0.5pt, dashed, color=black] table[x=1,y expr={1/\thisrowno{1}/4}]{\threeplfrect};
    \addplot[line width=0.5pt, dashdotted, color=black] table[x=1,y expr={1/\thisrowno{1}^1.5/4}]{\threeplfrect};
    \addplot[line width=0.5pt, densely dotted, color=black] table[x=1,y expr={1/\thisrowno{1}^2/4}]{\threeplfrect};
        
  \end{axis}
   \begin{axis}[
     name=plot2,
     at=(plot1.right of south east),
     anchor=left of south west,
     title ={ $m = k-1$},
     legend entries={$k=2$, $k=3$,$k=4$, $e_h$, $e_h^{\star}$, $h^2$, $h^3$,$h^4$},                            
     scale=0.6,
    legend style={text height=0.7em },
    legend style={draw=none},
    style={column sep=0.15cm},
    xlabel=number of elements $|\mathcal{T}|$,
    xmode=log,
    ymode=log,
    ytick = {1e-8,1e-6,1e-4,1e-2,1, 1e2, 1e4},
    y tick label style={
      /pgf/number format/.cd,
      fixed,
      precision=2
    },
    x tick label style={
      /pgf/number format/.cd,
      fixed,
      precision=2
    },
    %
    xmin = 1e1,
    xmax = 1e4,
    grid=both,
    legend style={
      cells={align=left},
      at={(1.05,1)},
      anchor = north west
    },
    ]

        \addlegendimage{only marks, mark=square*,black}
     \addlegendimage{only marks, mark=*,black}
    \addlegendimage{only marks, mark=triangle*,black}

    \addlegendimage{line width=0.5pt, myblue}
    \addlegendimage{line width=0.5pt, mygreen}
    \addlegendimage{dashed,line width=0.5pt}    
    \addlegendimage{dashdotted,line width=0.5pt}
    \addlegendimage{densely dotted,line width=0.5pt}

     \addplot[line width=0.5pt, color=mygreen, mark=*] table[x=1, y=2]{\threepltrect};
    \addplot[line width=0.5pt, color=mygreen, mark=triangle*] table[x=1, y=2]{\fourpltrect};
    \addplot[line width=0.5pt, color=mygreen, mark=square*] table[x=1, y=2]{\twopltrect};

    \addplot[line width=0.5pt, color=myblue, mark=*] table[x=1, y=2]{\threepltrecf};
    \addplot[line width=0.5pt, color=myblue, mark=triangle*] table[x=1, y=2]{\fourpltrecf};
    \addplot[line width=0.5pt, color=myblue, mark=square*] table[x=1, y=2]{\twopltrecf};
    
    \addplot[line width=0.5pt, dashed, color=black] table[x=1,y expr={1/\thisrowno{1}/4}]{\threeplfrect};
    \addplot[line width=0.5pt, dashdotted, color=black] table[x=1,y expr={1/\thisrowno{1}^1.5/4}]{\threeplfrect};
    \addplot[line width=0.5pt, densely dotted, color=black] table[x=1,y expr={1/\thisrowno{1}^2/4}]{\threeplfrect};
        
  \end{axis}
\end{tikzpicture}
\caption{Two dimensional test case as described in
  \cref{sec:numerical} using $\nu = 10^{-6}$. We plot the
  $H^1$-seminorm error of the velocity, against the number of elements
  in the mesh, for polynomial degrees $k=2,3,4$. Here $e_h = u-u_h$
  and $e_h^{\star}=u-u_h^{\star}$ with $u_h$ the discrete velocity
  solution to the EDG method \cref{eq:discrete_stokes} and
  $u_h^{\star}$ the discrete velocity solution to the modified EDG
  method \cref{eq:discrete_stokes_pr}. The solutions are computed both
  for $m=k$ in \cref{eq:approx_spaces_edg} (left) and when $m=k-1$
  (right).}
\label{fig:2dH1seminorm}
\end{figure}

\pgfplotstableread{nuconvergence_order_2_phatlo_False_rec_False.dat}  \nutwophfrf 
\pgfplotstableread{nuconvergence_order_2_phatlo_True_rec_False.dat}  \nutwophtrf  
\pgfplotstableread{nuconvergence_order_2_phatlo_False_rec_True.dat}  \nutwophfrt  
\pgfplotstableread{nuconvergence_order_2_phatlo_True_rec_True.dat}  \nutwophtrt   

\pgfplotstableread{nuconvergence_order_3_phatlo_False_rec_False.dat}  \nuthreephfrf 
\pgfplotstableread{nuconvergence_order_3_phatlo_True_rec_False.dat}   \nuthreephtrf 
\pgfplotstableread{nuconvergence_order_3_phatlo_False_rec_True.dat}   \nuthreephfrt 
\pgfplotstableread{nuconvergence_order_3_phatlo_True_rec_True.dat}    \nuthreephtrt 

\pgfplotstableread{nuconvergence_order_4_phatlo_False_rec_False.dat}  \nufourphfrf 
\pgfplotstableread{nuconvergence_order_4_phatlo_True_rec_False.dat}   \nufourphtrf 
\pgfplotstableread{nuconvergence_order_4_phatlo_False_rec_True.dat}   \nufourphfrt 
\pgfplotstableread{nuconvergence_order_4_phatlo_True_rec_True.dat}    \nufourphtrt 

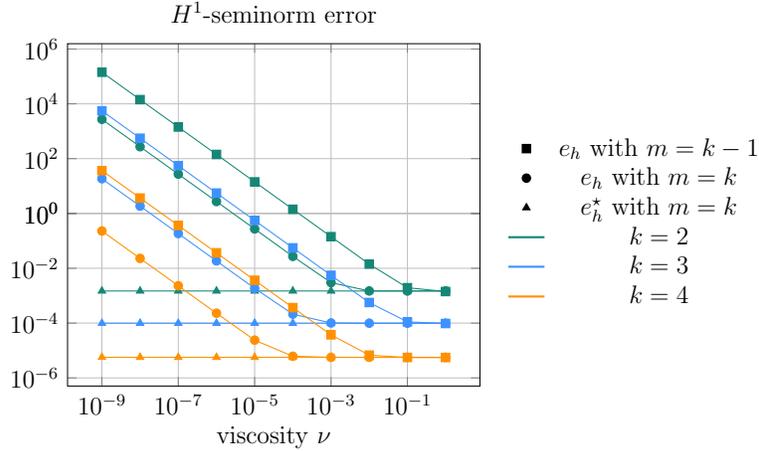
\begin{figure}
  \begin{center}
\begin{tikzpicture}
  [
  scale=0.8
  ]
  \begin{axis}[
    name=plot1,
    title={$H^1$-seminorm error},
    legend entries={$e_h \textrm{ with } m=k-1$, $e_h \textrm{ with } m=k$, $e_h^{\star} \textrm{ with }m=k$, $k=2$,$k=3$,$k=4$},
     scale=1,
    legend style={text height=0.7em },
    legend style={draw=none},
    style={column sep=0.15cm},
    xlabel=viscosity $\nu$,
    xmode=log,
    ymode=log,
    ytick = {1e-6, 1e-4,1,1e-2,1,1e2,1e4,1e6},
    y tick label style={
      /pgf/number format/.cd,
      fixed,
      precision=2
    },
    x tick label style={
      /pgf/number format/.cd,
      fixed,
      precision=2
    },
    %
    grid=both,
    legend style={
      cells={align=left},
      at={(1.05,0.75)},
      anchor = north west    },
    ]

    \addlegendimage{only marks, mark=square*}
    \addlegendimage{only marks, mark=*}
    \addlegendimage{only marks, mark=triangle*}

    \addlegendimage{line width=1pt, mygreen}
    \addlegendimage{line width=1pt, myblue}
    \addlegendimage{line width=1pt, myred}
    
    \addplot[line width=0.5pt, color=mygreen,mark=*] table[x=0,y=2]{\nutwophfrf};
    \addplot[line width=0.5pt, color=mygreen,mark=triangle*] table[x=0,y=2]{\nutwophfrt};
    \addplot[line width=0.5pt, color=mygreen,mark=square*] table[x=0,y=2]{\nutwophtrf};       

    \addplot[line width=0.5pt, color=myblue,mark=*] table[x=0,y=2]{\nuthreephfrf};
    \addplot[line width=0.5pt, color=myblue,mark=triangle*] table[x=0,y=2]{\nuthreephfrt};
    \addplot[line width=0.5pt, color=myblue,mark=square*] table[x=0,y=2]{\nuthreephtrf};

    \addplot[line width=0.5pt, color=myred,mark=*] table[x=0,y=2]{\nufourphfrf};
    \addplot[line width=0.5pt, color=myred,mark=triangle*] table[x=0,y=2]{\nufourphfrt};
    \addplot[line width=0.5pt, color=myred,mark=square*] table[x=0,y=2]{\nufourphtrf};
    
  \end{axis}
\end{tikzpicture}
  \end{center}
  \caption{Two dimensional test case as described in
    \cref{sec:numerical} using a fixed mesh with $|\mathcal{T}|=230$
    elements. We plot the $H^1$-seminorm error of the velocity against
    varying viscosities $\nu =1,\ldots,10^{-9}$, for polynomial
    degrees $k=2,3,4$. Here $e_h = u-u_h$ and
    $e_h^{\star}=u-u_h^{\star}$ with $u_h$ the discrete velocity
    solution to the EDG method \cref{eq:discrete_stokes} and
    $u_h^{\star}$ the discrete velocity solution to the modified EDG
    method \cref{eq:discrete_stokes_pr}. The solutions $u_h$ are
    computed both for $m=k$ and $m=k-1$.}
  \label{fig:2dnuH1seminorm}
\end{figure}

\begin{figure}  
\begin{tikzpicture}
  [
  scale=0.96
  ]
  \begin{axis}[
    name=plot1,
    scale=0.6,
    title = $L^2$-norm error,
    legend style={text height=0.7em },
    legend style={draw=none},
    style={column sep=0.1cm},
    ylabel=Error,
    xlabel=number of elements $|\mathcal{T}|$,
    xmode=log,
    ymode=log,
    ytick = {1e-7, 1e-5,1e-3,1e-1,10,1e3},
    y tick label style={
      /pgf/number format/.cd,
      fixed,
      precision=2
    },
    x tick label style={
      /pgf/number format/.cd,
      fixed,
      precision=2
    },
    %
    grid=both,
    legend style={
      cells={align=left},
      anchor = north west
    },
    ]

    \addplot[line width=0.5pt, color=mygreen, mark=*] table[x=1, y=3]{\tdoneplfrect};
    \addplot[line width=0.5pt, color=mygreen, mark=triangle*] table[x=1, y=3]{\tdtwoplfrect};

    \addplot[line width=0.5pt, color=myblue, mark=triangle*] table[x=1, y=3]{\tdtwoplfrecf};
    \addplot[line width=0.5pt, color=myblue, mark=*] table[x=1, y=3]{\tdoneplfrecf};

    \addplot[line width=0.5pt, dashed, color=black] table[x=1,y expr={1/\thisrowno{1}^0.66666/10}]{\tdoneplfrect};
    \addplot[line width=0.5pt, densely dotted, color=black] table[x=1,y expr={1/\thisrowno{1}/100}]{\tdoneplfrect};
    
  \end{axis}
   \begin{axis}[
     name=plot2,
     at=(plot1.right of south east),
     anchor=left of south west,
     title = $H^1$-seminorm error,
     legend entries={$k=1$, $k=2$, $e_h$, $e_h^{\star}$, $h$, $h^2$, $h^3$},                            
     scale=0.6,
    legend style={text height=0.7em },
    legend style={draw=none},
    style={column sep=0.15cm},
    xlabel=number of elements $|\mathcal{T}|$,
    xmode=log,
    ymode=log,
    ytick = {1e-4, 1e-2,1,1e2, 1e4},
    y tick label style={
      /pgf/number format/.cd,
      fixed,
      precision=2
    },
    x tick label style={
      /pgf/number format/.cd,
      fixed,
      precision=2
    },
    %
    grid=both,
    legend style={
      cells={align=left},
      at={(1.05,1)},
      anchor = north west
    },
    ]

     \addlegendimage{only marks, mark=*,black}
    \addlegendimage{only marks, mark=triangle*,black}
    \addlegendimage{line width=1pt, myblue}
    \addlegendimage{line width=1pt, mygreen}
    \addlegendimage{dashed,line width=0.5pt}
    \addlegendimage{dashdotted,line width=0.5pt}
    \addlegendimage{densely dotted,line width=0.5pt}

    \addplot[line width=0.5pt, color=mygreen, mark=*] table[x=1, y=2]{\tdoneplfrect};
    \addplot[line width=0.5pt, color=mygreen, mark=triangle*] table[x=1, y=2]{\tdtwoplfrect};

    \addplot[line width=0.5pt, color=myblue, mark=triangle*] table[x=1, y=2]{\tdtwoplfrecf};
    \addplot[line width=0.5pt, color=myblue, mark=*] table[x=1, y=2]{\tdoneplfrecf};

    \addplot[line width=0.5pt, dashed, color=black] table[x=1,y expr={1/\thisrowno{1}^0.33333/15}]{\tdoneplfrect};
    \addplot[line width=0.5pt, dashdotted, color=black] table[x=1,y expr={1/\thisrowno{1}^0.66666/10}]{\tdoneplfrect};
  \end{axis}
\end{tikzpicture}
\caption{Three dimensional test case as described in
  \cref{sec:numerical} using $\nu = 10^{-6}$. We plot the $L^2$-norm
  (left) and $H^1$-seminorm (right) errors of the velocity, against
  the number of elements in the mesh, for polynomial degrees
  $k=1,2$. Here $e_h = u-u_h$ and $e_h^{\star}=u-u_h^{\star}$ with
  $u_h$ the discrete velocity solution to the EDG method
  \cref{eq:discrete_stokes} and $u_h^{\star}$ the discrete velocity
  solution to the modified EDG method
  \cref{eq:discrete_stokes_pr}. The solutions are computed using $m=k$
  in \cref{eq:approx_spaces_edg}.}
\label{fig:3derrors}
\end{figure}

\section{Conclusions}
\label{sec:conclusions}

We introduced a new reconstruction operator that restores
pressure-robustness for an embedded discontinuous Galerkin
discretization of the Stokes equations. We have shown that this
reconstruction operator can be constructed locally on vertex patches
and needs to be applied only to the right hand side vector, allowing
for easy implementation in existing codes. Furthermore, by an a priori
error analysis, we showed that the velocity errors converge
optimally. Numerical examples in two and three dimensions support our
analysis.

\bibliographystyle{siamplain}
\bibliography{references}
\end{document}